\documentclass[11 pt]{article}
\usepackage{amsmath}
\usepackage{amssymb}
\usepackage{amsfonts}
\usepackage{amscd}
\usepackage{amsthm}
\usepackage[all,cmtip]{xy}
\usepackage{tikz-cd}
\usepackage{pinlabel}
\usepackage[vmargin=0.95in, hmargin=1.5in]{geometry}

\newtheorem{theorem}{Theorem}[section]
\newtheorem{lemma}[theorem]{Lemma}
\newtheorem{cor}[theorem]{Corollary}
\newtheorem{prop}[theorem]{Proposition}

\theoremstyle{remark}
\newtheorem{definition}{Definition}[section]
\newtheorem{remark}{Remark}[section]
\newtheorem{claim}{Claim}[section]

\newcommand{\mcO}{\mathcal O}
\newcommand{\mcN}{\mathcal N}
\newcommand{\mcS}{\mathcal S}
\newcommand{\mc}{\mathcal}

\DeclareMathOperator{\SL}{SL}
\DeclareMathOperator{\HH}{H}
\DeclareMathOperator{\nrd}{nrd}
\DeclareMathOperator{\Prd}{Prd}
\DeclareMathOperator{\lk}{lk}
\DeclareMathOperator{\st}{st}

\DeclareMathOperator{\Sym}{Sym}
\DeclareMathOperator{\Alt}{Alt}
\DeclareMathOperator{\End}{End}
\DeclareMathOperator{\id}{id}
\DeclareMathOperator{\MCG}{MCG}

\newcommand\Q{\text{$\mathbb{Q}$}}
\newcommand\Z{\text{$\mathbb{Z}$}}
\newcommand\C{\text{$\mathbb{C}$}}
\newcommand\N{\text{$\mathbb{N}$}}
\newcommand\Imrho{\text{Im$(\tilde \rho)$}}

\DeclareMathOperator{\Aut}{Aut}
\DeclareMathOperator{\Out}{Out}
\DeclareMathOperator{\Inn}{Inn}
\DeclareMathOperator{\GL}{GL}
\DeclareMathOperator{\Mat}{Mat}

\title{Arithmetic Quotients of the Automorphism Group of a Right-Angled Artin Group}
\author{Justin Malestein\thanks{University of Oklahoma, Department of Mathematics, 601 Elm Ave, Norman, OK, 73019, justin.malestein@ou.edu}}

\begin{document}
	\maketitle
	
	\begin{abstract}
		It was previously shown by Grunewald and Lubotzky that the automorphism group of a free group, $\Aut(F_n)$, has a large collection
		of virtual arithmetic quotients. Analogous results were proved for the mapping class group by Looijenga and by Grunewald, Larsen,
		Lubotzky, and Malestein. In this paper, we prove analogous results for the automorphism group of a right-angled Artin group
		for a large collection of defining graphs. As a corollary of our methods we produce new virtual arithmetic quotients
		of $\Aut(F_n)$ for $n \geq 4$ where $k$th powers of all transvections act trivially for some fixed $k$.
		Thus, for some values of $k$, we deduce that the quotient of $\Aut(F_n)$ by the subgroup generated by $k$th powers
		of transvections contains nonabelian free groups. This expands on results of Malestein and Putman and of Bridson and Vogtmann.
	\end{abstract}

\section{Introduction}
It is well-known that $\Aut(F_n)$, the automorphism group of the free group, has a surjective representation onto $\GL_n(\Z)$ 
from the action on the abelianization of $F_n$. One can define similar representations via the action on
$\HH_1(R) = R/[R, R]$ where $R < F_n$ is a normal subgroup of finite index. 
In \cite{GL}, Grunewald and Lubotzky determined the image of the (virtual) representation
$\Aut(F_n) \to \Aut(\HH_1(R))$ when $R$ contains a primitive element and thereby produced a large collection of
virtual arithmetic quotients. Here, we say virtual since the representation may only be defined on a 
finite index subgroup of $\Aut(F_n)$ and is surjective up to finite index. As a sampling of their results, for any $n \geq 4$, they obtain 
$\SL_{m(n-1)}(\Z[\zeta])$ as a virtual quotient where $m$ is any positive integer and $\zeta$ is any root of unity.

Recently, there has been much interest in extending results about $\Aut(F_n)$ (or $\Out(F_n)$) 
to $\Aut(A_\Gamma)$ (or $\Out(A_\Gamma)$) or finding analogous results for $\Aut(A_\Gamma)$
where $A_\Gamma$ is a right-angled Artin group (RAAG).
Recall that for any finite graph $\Gamma$, the corresponding right-angled Artin group $A_\Gamma$
is the group with presentation
$$ \langle v \in V(\Gamma) \;\;\; | \;\;\; [v, w] \text{ if } v, w \text{ are adjacent in } \Gamma \rangle.$$
In the case where $\Gamma$ is an independent set, $A_\Gamma = F_n$.
In this paper, we show that, for a large class of graphs, $\Aut(A_\Gamma)$
has a rich collection of virtual arithmetic quotients via actions on the homology of finite index subgroups of $A_\Gamma$, and in the process, 
we understand better
these representations in the case of $\Aut(F_n)$ for some finite index $R < F_n$ containing no primitive elements. 

\begin{remark}
For arbitrary $\Gamma$, $\Out(A_\Gamma)$ need not have any ``interesting'' quotients since, e.g., there are $\Gamma$ for which $\Out(A_\Gamma)$ is finite
and, in one model of randomness, this is the generic case \cite{CharneyFarber, Day}. Therefore, some conditions on $\Gamma$ are needed, but the minimum
conditions are not at all clear. The conditions we impose are described below.
\end{remark}

\subsection{New results for $\Out(F_n)$}

As a corollary of our results, we obtain new virtual arithmetic quotients of $\Aut(F_n)$
different from those in \cite{GL}.
Whereas Grunewald and Lubotzky study the action on $\HH_1(R)$ for $R < F_n$ redundant (i.e. $R$ contains a primitive element),
we are able to describe the action on part of $\HH_1(R)$ for some (but not all) nonredundant $R$.
Our new results combined with those of \cite{MPut} yield the following theorem.

\begin{theorem} \label{thm:outfnmodtrans}
	For each $n \geq 4$, there is an infinite set $\mc K_n$ such that for every $k \in \mc K_n$,
	there is a virtual arithmetic quotient of $\Out(F_n)$ whose kernel contains
	the $k$th powers of all transvections. Moreover, that quotient contains nonabelian free subgroups.
\end{theorem}

\begin{remark}
	The kernel contains powers of all transvections relative to any basis of $F_n$.
\end{remark}

\begin{remark}
	The virtual arithmetic quotients are of the form $\SL_{n-1}(\mcO)$ where $\mcO$ is an order
	in a finite-dimensional simple algebra over $\Q$. However, we do not determine $\mcO$ explicitly.
\end{remark}

As noted in \cite{MPut}, we can take $\mc K_n$ to be 
$$ \mc K_n = \{k \;\;|\;\; \exists \text{prime power } p^e \text{ dividing } k 
	\text{ where } p^e > p(p - 1)(n - 1)\}. $$ 
This is an improvement on 
Theorem D of \cite{MPut} which states that $\Out(F_n)$ mod the subgroup generated
by $k$th powers of transvections contains elements of infinite order
(which itself is an improvement of a theorem of Brison and Vogtmann \cite{BridsonVogtmann}
that this quotient group is infinite).

\subsection{Domination and the homology representation of $\Aut(A_\Gamma)$}
To give some context to our results, we first describe the virtual image under the homology
representation $\rho_0 \colon \Aut(A_\Gamma) \to \Aut(\HH_1(A_\Gamma))$.
Since the only relations in $A_\Gamma$ are commutation relations, we see that 
$\HH_1(A_\Gamma) = A_\Gamma/[A_\Gamma, A_\Gamma] \cong \Z^n$ where $n = |V(\Gamma)|$.
Thus, $\rho_0$ can be viewed as a representation to $\GL_n(\Z)$. To describe the virtual image,
we must define domination of vertices and the generators of $\Aut(A_\Gamma)$.

We say that a vertex $v \in V(\Gamma)$ {\it dominates} $w \in V(\Gamma)$ if the link of $w$ ($\lk(w)$)
is a subset of the star of $v$ ($\st(v)$).
We denote this relation by $v \geq w$. The domination relation is transitive and reflexive; i.e. it is a 
{\it preorder} \cite{CharneyVogtmann}[Lemma 2.2]. 
We say that $v$ is {\it domination equivalent} to $w$ if $v$ dominates $w$ and $w$ dominates $v$. This 
is an equivalence relation, and domination induces a poset on domination equivalence classes.

For a vertex $v$, let $U(v) = \{w \in V(\Gamma) \mid w \geq v\}$. Using vertices as a basis
of $\HH_1(A_\Gamma)$, let $\langle U(v) \rangle$ denote the free summand generated by $U(v)$.
Then $\rho_0(\Aut(A_\Gamma))$ is commensurable with the subgroup of $\GL_n(\Z)$ preserving
the subspaces $\langle U(v) \rangle$ for all $v \in V(\Gamma)$. (This is proven implicitly
in \cite{Day2}[Corollary 3.11].) Domination posets need not
be a chain, but, relative to some basis, the subgroup preserving all $\langle U(v) \rangle$ is
block uppertriangular where some blocks above the diagonal are $0$. See the appendix for some examples.

To see that this is the virtual image, we use the generators of $\Aut(A_\Gamma)$.
By a result of Laurence, $\Aut(A_\Gamma)$ is generated by the following types of automorphisms\cite{Laurence}.
\begin{itemize}
	\item \textbf{Dominated Transvections:} Given $v \geq w$, a corresponding dominated transvection
	maps $w$ to $wv$ or $vw$ and fixes all other vertices.
	\item \textbf{Partial Conjugations:} Given a component $C$ of $\Gamma - \st(v)$, the corresponding
	partial conjugation maps $w \mapsto vwv^{-1}$ for all $w \in C$ and fixes all other vertices.
	\item \textbf{Inversions:} These invert a generator $v$.
	\item \textbf{Graphic Automorphisms:} These are automorphisms induced by graph automorphisms of $\Gamma$.
\end{itemize}
As can be shown, (see Lemma \ref{lemma:onlytandpc}),
the subgroup, $\Aut^0(A_\Gamma)$ generated by transvections, partial conjugations, and inversions
is finite index in $\Aut(A_\Gamma)$. Partial conjugations act trivially on $\HH_1(A_\Gamma)$, and inversions act by 
diagonal matrices with entries $\pm 1$. Homologically, a transvection only adds to some $w \in V(\Gamma)$
a multiple of some $u \in V(\Gamma)$ which dominates $w$ and therefore it preserves the subspaces
$\langle U(v) \rangle$. In fact, transvections map to elementary matrices, and these and the diagonal matrices 
can be shown to generate the subgroup of all invertible matrices preserving the subspaces $\langle U(v) \rangle$.

\subsection{Main Theorems}
As noted above, one can only produce virtual arithmetic quotients when $\Aut(A_\Gamma)$ has some complexity, 
and so we impose some conditions on $\Gamma$.
We will present these conditions shortly, but first we describe the arithmetic quotients 
obtained under those conditions. The arithmetic quotients will look similar to the image of $\rho_0$ in that
there will be invariant subspaces analogous to $\langle U(v) \rangle$. Some key differences are that there
is an additional invariant subspace and the image lies not in $\GL_n(\Z)$ but $\GL_{n-1}(\mcO)$ where $\mcO$ is a ring of matrices.

To an order $\mcO$ in a finite-dimensional simple $\Q$-algebra, a preorder $\preceq$ on a finite set $W$, and a 
submodule $H_\mcO$ of the free left $\mcO$-module on $W$, we will associate two arithmetic groups $\tilde{\mc G}$ and $\mc G$ defined as follows. An \textit{order} is a subring which spans the $\Q$-algebra as a $\Q$-vector space
and is isomorphic to a free $\Z$-module. E.g. $\Mat_m(\Z) \subset \Mat_m(\Q)$.
Let $C_\mcO$ be the free left $\mcO$-module on $W$, and for any subset $W' \subseteq W$,
let $C_{\mcO}[W']$ be the left $\mcO$-submodule of $C_{\mcO}$ generated by $W'$.
Abusing notation\footnote{The notation $U(v)$ and $L(v)$ will denote the upper bounds and lower bounds resp. of an element $v$ 
relative to the preorder that makes sense in
that context. The preorder will always be clear except for one instance in Section \ref{section:reduction} where
we will specify the preorder. Equivalence classes will also refer to equivalence classes induced by the preorder, i.e. $v \sim w$ if 
$v \succeq w$ and $v \preceq w$.}, we let $U(v)$ be the upper bounds of $v$ under $\preceq$, i.e. $U(v) = \{w \in W \mid w \succeq v\}$.
Let $U'(v) = U(v) - v$.
We define $\tilde{\mc G} < \Aut_{\mcO}(C_{\mcO})$ to be the subgroup of automorphisms
satisfying the following conditions.
\begin{enumerate}
	\item $C_{\mcO}[U(v)]$ is invariant for all $v \in W$.
	\item The restriction to $C_{\mcO}[U(v)]$ has reduced norm $1$.
	\item If $v$ has a trivial equivalence class, then the action on $C_{\mcO}[U(v)]/C_{\mcO}[U'(v)]$ 
	is trivial.
	\item $H_{\mcO}$ is invariant.
	\item The action on $C_{\mcO}/H_{\mcO}$ is trivial.
\end{enumerate}
We define $\mc G$ to be the restriction of $\tilde{\mc G}$ to $H_{\mcO}$.
See section \ref{section:linalggrp} for the definition of reduced norm. In the case that
$\mcO = \Mat_m(\Z[\zeta])$ where $\zeta$ is a root of unity, reduced norm is the same as the determinant
under a canonical identification $\GL_{|U(v)|}(\mcO^{op}) \cong \GL_{m|U(v)|}(\Z[\zeta])$.
While it's convenient later to allow $H_{\mcO}$ to be an arbitrary submodule, we let
$H_{\mcO}^-$ be the submodule generated by all differences of all pairs of free generators.

Note that if $\mcO = \Z$, $W = V(\Gamma)$, and $\preceq$ is the domination preorder, then the image of $\rho_0$ virtually satisfies 
conditions 1, 2, and 3. The only essentially new conditions are 4 and 5,
(and the fact that $W$ is usually a proper subset of $V(\Gamma)$).

\begin{theorem} \label{thm:mainthm}
	Let $\Gamma$ be a finite graph,  $W \subseteq V(\Gamma)$ a subset satisfying conditions $\dagger$ defined below,
	and $\preceq$ the domination preorder restricted to $W$.
	Let $\mcO = \Mat_m(\Z[\zeta])$ where $\zeta$ is a primitive $k$th root of unity, and let $H_\mcO = H_\mcO^- \subseteq C_\mcO$.
	Then, there is a virtual surjective representation from $\Out(A_\Gamma)$ to the group $\mc G$ defined above
	if one of the following conditions is satisfied.
	\begin{itemize}
	\item $m \geq 6|W|+2$ and $k = 1$.
	\item $1 < m < k$ divides $k$, the largest prime factor of $k$ is at least $|W|$, 
	and $m, \frac{k}{m}$ are coprime.
	\end{itemize}
\end{theorem}

\noindent In the appendix, we present some example arithmetic quotients as matrix groups.

\begin{remark}
Our methods can produce even more arithmetic quotients, but we defer the most general statement
we can make to a later section. See Corollary \ref{cor:generaloutagamquots}.
\end{remark}

In the course of the proofs, we will define a virtual representation from $\Aut(A_\Gamma)$ to $\Aut_{\mcO}(C_{\mcO})$ as well. 
The reason for restricting to $H_{\mcO}$ is that
the image in $\Aut_{\mcO}(C_{\mcO})$, in general, need not lie in $\tilde{\mc G}$, but their restrictions to $H_{\mcO}$ coincide.
Nevertheless, we can determine the image inside $\Aut_{\mcO}(C_{\mcO})$ (up to isomorphism). 

\begin{theorem} \label{thm:mainthmsemidprod}
	Let $\Gamma, W, \preceq, \mc O, H_\mcO,$ and $\mc G$ be as in Theorem \ref{thm:mainthm}.
	Then, there is a virtual surjective representation from $\Aut(A_\Gamma)$ to $\mc G \ltimes H_\mcO$
	if one of the following conditions is satisfied.
	\begin{itemize}
	\item $m \geq 6|W|+2$ and $k = 1$.
	\item $1 < m < k$ divides $k$, the largest prime factor of $k$ is at least $|W|$, 
	and $m, \frac{k}{m}$ are coprime.
	\end{itemize}	
\end{theorem}

As we will see, we don't have a representation from $\Out(A_\Gamma)$ to $\mc G \ltimes H_{\mcO}$ 
since inner automorphisms act nontrivially and (virtually) surject onto the $H_{\mcO}$ factor.
In the case $A_\Gamma = F_n$ and $n \geq 4$, the image in $\Aut_{\mcO}(C_{\mcO})$ turns out to be the same as 
$\tilde{\mc G}$. Some example arithmetic quotients of $\Aut(F_n)$ that we obtain are
$\SL_{(n-1)m}(\Z) \ltimes \Z^{(n-1)m}$ for all $n \geq 4$ and $m \geq 6n+2$. (See Remark \ref{remark:mfGistildeG}.)

\subsection{Conditions on $\Gamma$}
We now define the conditions $\dagger$. Let $L(v) = \{w \in V(\Gamma) \mid w \leq v\}$.
We say that \textit{$v \in V(\Gamma)$ divides $W \subseteq V(\Gamma)$ trivially} if 
$W - L(v)$ lies in a single component of $\Gamma - \st(v)$. \\

\noindent \textbf{Conditions:} The conditions $\dagger$ for a subset $W \subseteq V(\Gamma)$ are the following.
\begin{enumerate}
	\item $W$ is an independent set in $\Gamma$.
	\item There exist two distinct vertices in $W$ both dominating a third vertex in $W$.
	\item Maximal domination equivalence classes in $W$ do not have exactly $2$ vertices.
	\item All vertices in $W$ divide $W$ trivially.
	\item $W$ is closed under lower bounds, 
		i.e. if $v_1 \leq v_2$ and $v_2 \in W$, then $v_1 \in W$.
\end{enumerate}

In Section \ref{section:conditions}, we will discuss how restrictive these conditions are
and also how necessary they are to obtain virtual arithmetic quotients. 
While the conditions $\dagger$ seem rather restrictive, we will see that a graph
lacking such a vertex subset has some strong constraints. Note that condition 2
ensures that $\mc G$ is nontrivial.

We are mainly focused on arithmetic quotients
generated by images of transvections. Moreover, some of the conditions above
are imposed to ensure that partial conjugations act essentially trivially
under the representation. Specifically, condition 4 implies that any partial conjugation of $W$
by a vertex in $W$ is, up to an inner automorphism, equivalent to a product
of dominated transvections (if we ignore what happens to vertices outside $W$).
While there are almost certainly $\Aut(A_\Gamma)$ with interesting actions on some
$\HH_1(R)$ where partial conjugations act nontrivially, we were unable to prove
any kind of general theorem in such cases. Indeed, at the far extreme, we could not
determine the virtual image of $\Aut(A_\Gamma)$ in $\Aut(\HH_1(R))$ when $\Aut(A_\Gamma)$
admits partial conjugations but no dominated transvections.

\subsection{Relative automorphism groups}
In \cite{DayWade}, Day and Wade study relative automorphism groups of right-angled Artin groups.
The proof of our theorems proceed by first reducing the main theorems to a similar statement
about relative automorphism groups of free groups. As a corollary of our methods, we obtain
arithmetic quotients of certain large classes of relative outer automorphism groups of $F_n$.
Specifically, Theorems \ref{thm:mainthm} and \ref{thm:mainthmsemidprod} apply if we replace
$\Aut(A_\Gamma)$ and $\Out(A_\Gamma)$ with the relative automorphism groups that appear in
Section \ref{section:reduction}. Precise theorems are stated in Section \ref{section:definingrho}.

\subsection{Other related work}
Guirardel and Sale use a particular virtual representation similar to the representations studied in this paper
to show that $\Out(A_\Gamma)$ is large (i.e. virtually maps onto a nonabelian free group) 
when $\Gamma$ has a simply intersecting link of a particular type \cite{GuirardelSale}.
Since their goal is quite different, they only consider the action of $\Aut(A_\Gamma)$ on
$\HH_1(R)$ for a particular index $2$ subgroup $R < A_\Gamma$.
While some of the work in this paper could be used to prove, in some cases, largeness or other
properties considered by Guirardel and Sale, any such cases
have already been covered by their paper \cite{GuirardelSale}.

One can try to prove similar results as the above for $\Out(\pi_1(\Sigma))$ where $\Sigma$ is a closed, orientable surface.
In this case, $\Out(\pi_1(\Sigma))$ is virtually the mapping class group $\MCG(\Sigma)$ of $\Sigma$. It was shown by Looijenga,
in the case where $\pi_1(\Sigma)/R$ is finite abelian, and then later by Grunewald, Larsen, Lubotzky, and Malestein, for more
general finite index $R$ (but not all finite index $R$), that the virtual action on $\HH_1(R)$ yields virtual arithmetic
quotients of the mapping class group \cite{GLLM, Looijenga}. A number of other papers
have also investigated these virtual linear representations of $\Aut(F_n)$ and $\MCG(\Sigma)$
for a variety of purposes \cite{FarbHensel, Hadari1, Hadari2, Hadari3, Hadari4, Koberda, Liu, McMullen, PW, Sun}.

\subsection{Outline of paper}
In Section \ref{section:reduction}, we reduce the main theorems to statements about relative automorphism groups of $F_n$.
In Section \ref{section:definingrho}, we define the virtual representation
of the relative automorphism groups and present the main technical theorems and corollaries of the paper. In Sections
\ref{section:upperbound} and \ref{section:lowerbound}, we prove the main technical theorems.
In Section \ref{section:fingrps}, we show that certain finite groups have irreducible rational representations
with the required properties. In section \ref{section:proofofmainthms}, we finish the proof of Theorems \ref{thm:outfnmodtrans},
\ref{thm:mainthm}, and \ref{thm:mainthmsemidprod}.
In Section \ref{section:conditions}, we discuss conditions $\dagger$, and in the appendix, we describe some examples of
$\mc G$ as groups of matrices.

\section{Reducing to subgroups of $\Aut(F_n)$} \label{section:reduction}
In this section, we reduce the proof of the main theorem to statements about relative automorphism groups of $F_n$. While
this is not strictly necessary, it will significantly reduce the amount of technical details in the proof
which still are considerable.
By doing so, we will also produce arithmetic quotients of these relative automorphism groups.
We first start with an elementary lemma so that we may ignore permutation automorphisms.
Recall that $\Aut^0(A_\Gamma) < \Aut(A_\Gamma)$ is the subgroup generated by all dominated transvections, partial
conjugations and inversions. While the next lemma is elementary, we provide a proof for convenience.

\begin{lemma} \label{lemma:onlytandpc}
	The subgroup $\Aut^0(A_\Gamma)$ is of finite index in $\Aut(A_\Gamma)$.
\end{lemma}
\begin{proof}
	It is easy to check that the conjugate of a transvection (resp. partial conjugation or inversion) by a
	graph automorphism is a transvection (resp. partial conjugation or inversion). Consequently, $\Aut^0(A_\Gamma)$ is a normal
	subgroup, and the quotient $\Aut(A_\Gamma)/\Aut^0(A_\Gamma)$ is generated by graph automorphisms which
	must be finite.
\end{proof}

Since we are only interested in virtual representations of $\Aut(A_\Gamma)$, we can focus instead on
$\Aut^0(A_\Gamma)$. The next lemma implies that the action of $\Aut^0(A_\Gamma)$ descends naturally
to a quotient of $A_\Gamma$. For a subset $X \subseteq A_\Gamma$, let $\langle \langle X \rangle \rangle$ denote
the normal subgroup generated by $X$. Results similar to this lemma were proven in \cite{GuirardelSale}
but for different choices of $W$.

\begin{lemma} \label{lemma:Gam-Pischar}
 	Suppose $W$ is closed under lower bounds. 
	Then, $\langle \langle \Gamma-W \rangle \rangle$ is invariant under $\Aut^0(A_\Gamma)$.
\end{lemma}

\begin{proof}
	Since $W$ is closed under lower bounds, the complement is closed under upper bounds.
	It then follows easily that $\langle \langle \Gamma - W \rangle \rangle$ is invariant under
	dominated transvections. Showing $\langle \langle \Gamma - W \rangle \rangle$ is invariant under
	any partial conjugation or inversion is straightforward.
\end{proof}

When $W$ satisfies $\dagger$,
the above lemma tells us that $\Aut^0(A_\Gamma)$ descends to an action on $A_\Gamma/\langle \langle \Gamma - W \rangle \rangle \cong F[W]$
where we use $F[W]$ to denote the free group on generators $W$.
We will see that the image of $\Aut^0(A_\Gamma)$ in $\Aut(F[W])$ is, up to inner automorphisms, a relative 
automorphism group of $F[W]$.

\subsection{Relative automorphism groups and their generators}
Consider an arbitrary finite set $W$ and a preorder $\preceq$ on $W$. While we will be interested
in the case when $W$ is a subset of $V(\Gamma)$ and $\preceq$ is the restriction of the domination relation,
the results in this section hold generally.
For $U \subseteq W$, let $F[U]$ be the corresponding free factor of $F[W]$. Let
$$\mc U = \{F[U] \mid U \subseteq W \text{ is closed under upper bounds under } \preceq \}.$$
We let $\Aut(F[W], \mc U)$ be the subgroup of automorphisms mapping each $F[U] \in \mc U$ to itself.
We call this the automorphism group relative to $\mc U$.

In \cite{DayWade}, Day and Wade define relative outer automorphism groups of 
RAAGs. We will only use their results in the case of free groups. They define $\Out^0(F[W])$ to be 
the subgroup generated by inversions, transvections, and partial conjugations, but $F[W]$ is a free group, so
$\Out^0(F[W]) = \Out(F[W])$.
The relative outer automorphism group, denoted $\Out^0(F[W], \mc U)$ is the subgroup of $\Out^0(F[W])$
where each outer automorphism has a representative automorphism in $\Aut(F[W])$ which maps
$F[U]$ to itself, but the representative may depend on $F[U]$. A priori, this seems
different from the group $\Aut(F[W], \mc U)$ we define above which leaves invariant all $F[U]$
simultaneously without conjugation. However,
in this case where $\mc U$ is defined as above from a preorder, we will show
that $\Aut(F[W], \mc U)$ projects to $\Out^0(F[W], \mc U)$. Recall that
$L(v) = \{w \in W \;\; | \;\; v \succeq w\}$.

\begin{lemma} \label{lemma:relautgenTI}
	Let $W$ be a set with a preorder $\preceq$, and let $\mc U$ be defined as above.
	Then, the image of $\Aut(F[W], \mc U)$ in $\Out(F[W])$ is $\Out^0(F[W], \mc U)$. Moreover,
	$\Aut(F[W], \mc U)$ is generated by the transvections and inversions contained in it.
\end{lemma}
\begin{proof}
	We first show that $\Out^0(F[W], \mc U)$ is generated by the transvections and inversions
	contained in it. By Theorem D of \cite{DayWade}, $\Out^0(F[W], \mc U)$ is generated by the
	transvections, inversions, and partial conjugations contained in it, so it suffices to  
	show that any partial conjugation in $\Out^0(F[W], \mc U)$ is a product of transvections.
	
	Let $v \in W$. If $L(v) = W$, then all vertices may be transvected by $v$, and so 
	any partial conjugation by $v$ is a product of transvections by $v$. Suppose then
	that $L(v) \neq W$, and $\phi \in \Out^0(F[W], \mc U)$ is a partial conjugation by $v$.
	Let $u \in W - L(v)$, and let $\varphi \in \Aut(F[W])$ be a representative
	such that $\varphi(u) = u$. We claim that $\varphi(u') = u'$ for all $u' \in W - L(v)$.
	If not, then there is some $u' \in W - L(v)$ such that $\varphi(u') = v^{\pm 1} u' v^{\mp 1}$.
	Then, $U(u) \cup U(u')$ is closed under upper bounds, and $F[U(u) \cup U(u')]$ is not
	invariant under $\varphi$ or any element of $\Inn(F[W]) \cdot \varphi$. This contradicts
	the fact that $\phi \in \Out^0(F[W], \mc U)$.
	
	Let $\mc {TI} < \Aut(F[W], \mc U)$ be the subgroup
	generated by transvections and inversions in $\Aut(F[W], \mc U)$. 
	From the above, we see that the image of both $\mc{TI}$ and $\Aut(F[W], \mc U)$
	in $\Out(F[W])$ is $\Out^0(F[W], \mc U)$. 
	
	We now show $\mc{TI} = \Aut(F[W], \mc U)$.
	Suppose $\psi \in \Aut(F[W], \mc U)$. Then, up to an inner automorphism, it also lies in $\mc{TI}$, and so
	there is some $x \in F[W]$ and $\varphi \in \mc{TI}$ such that $\psi = C_x \circ \varphi$
	where $C_x$ is the inner automorphism conjugating by $x$. Then $C_x = \psi \circ \varphi^{-1} \in \Aut(F[W], \mc U)$.
	Since $C_x$ preserves all free factors $F[U] \in \mc U$, we have $\displaystyle x \in \bigcap_{U \in \mc U} F[U]$.
	I.e. $x$ is the identity or $x \in F[U']$ for some subset $U' \subseteq W$ where $L(u') = W$ for all $u' \in U$.
	In either case, $C_x$ is a product of transvections
	in $\mc{TI}$, and so $\psi \in \mc{TI}$.
\end{proof}

\subsection{Reduction to the relative automorphism group}

Now that we've defined the relative automorphism group of a free group, we would like to reduce our main theorems
to theorems about this group. We first reduce to a closely related group.

\begin{lemma} \label{lemma:AutAGamtoAutFn}
	Suppose $W \subseteq V(\Gamma)$ satisfies $\dagger$. Let the preorder $\preceq$ on $W$ be the restriction of the 
	domination relation to $W$. Let $\mc U$ be as defined above. Then, the image of $\Aut^0(A_\Gamma)$ in 
	$\Aut^0(A_\Gamma/\langle \langle \Gamma - W \rangle \rangle) = \Aut(F[W])$ is $\Aut(F[W], \mc U) \cdot \Inn(F[W])$.
\end{lemma}

Note that here there are two preorders. To remove the ambiguity, we specify that, 
in the proof below, $L(v)$ denotes the lower bounds in $V(\Gamma)$ of $v$ under the domination relation.

\begin{proof}
	First, we observe that the image of $\Aut^0(A_\Gamma)$ contains $\Aut(F[W], \mc U) \cdot \Inn(F[W])$.
	Clearly, $\Inn(F[W])$ is the image of $\Inn(A_\Gamma)$. Moreover, by Lemma \ref{lemma:relautgenTI},
	$\Aut(F[W], \mc U)$ is generated by inversions and the transvections contained it. The inversions
	are clearly in the image of $\Aut^0(A_\Gamma)$. Since $\mc U$
	is ultimately defined via the domination relation restricted to $W$, transvections in $\Aut(F[W], \mc U)$ come from 
	corresponding dominated transvections in $\Aut^0(A_\Gamma)$. 
	
	To prove that the image of $\Aut^0(A_\Gamma)$ is contained in $\Aut(F[W], \mc U) \cdot \Inn(F[W])$, we analyze case
	by case the action of the generators of $\Aut^0(A_\Gamma)$ on $F[W] = A_\Gamma/\langle \langle \Gamma - W \rangle \rangle$. Clearly, 
	transvections, inversions or partial conjugations by vertices in $\Gamma - W$ act trivially. Transvections or inversions
	by elements in $W$ map to the corresponding transvections or inversions in $\Aut(F[W], \mc U)$.
	
	It remains to analyze the action of a partial conjugation by an arbitrary vertex $v \in ßW$.
	Since $W$ satisfies $\dagger$, any such $v$ divides $W$ trivially, and thus either
	conjugates a single $u \in L(v) \subseteq W$, no vertex of $W$, or all vertices of $W - L(v)$. 
	In the first case, $\varphi$ is a composition of two dominated transvections of $u$ by $v$ which are both in $W$. 
	In the second case, $\varphi$ acts trivially.
	
	Suppose instead we are in the third case. Then $C_{v^{-1}} \circ \varphi$, where $C_{v^{-1}} \in \Aut(A_\Gamma)$ is the inner
	automorphism by $v^{-1}$, conjugates all of $L(v)$ by $v^{-1}$ and fixes all vertices in $W - L(v)$.
	Thus, $C_{v^{-1}} \circ \varphi$ is a composition of dominated transvections of all $u \in L(v) - v \subseteq W$ 
	by $v \in W$. Moreover, $C_{v^{-1}}$ acts as an inner automorphism on $A_\Gamma/\langle \langle \Gamma - W \rangle \rangle = F[W]$, and 
	dominated transvections in $\Aut(A_\Gamma)$ map
	to the corresponding transvection in $\Aut(F[W], \mc U)$. Since the images of $C_{v^{-1}} \circ \varphi$ and
	$C_{v^{-1}}$ lie in $\Aut(F[W], \mc U) \cdot \Inn(F[W])$, so must the image of $\varphi$.
\end{proof}

Keeping the inner automorphisms around would be unwieldy later, so we show that, for the representations of interest,
we may dispense with them. In order to do that, we must take a first step towards defining the representation. Let $R < F[W]$
be a finite index normal subgroup. Since $R$ need not be invariant under all automorphisms, we define the following finite index
subgroup of $\Aut(F[W])$ where $q: F[W] \to F[W]/R$ is the quotient map 
$$\Aut(F[W]; R) = \{  \varphi \in \Aut(F[W]) \;\; | \;\; q \circ \varphi = q \}.$$
We let $\Aut(F[W], \mc U; R) = \Aut(F[W], \mc U) \cap \Aut(F[W]; R)$. Note that this may be a proper subgroup
of the group which merely preserves $R$. Our first representation is then $\nu: \Aut(F[W];R) \to \Aut(\HH_1(R))$
defined by restriction to $R$ and passing to the quotient $R/[R,R] = \HH_1(R)$.

\begin{lemma} \label{lemma:transvonly}
	Let $W$ be a finite set, $\preceq$ a preorder on $W$, $\mc U$ as above, and $\nu$ as above.
	Let $\Delta = (\Aut(F[W], \mc U) \cdot \Inn(F[W])) \cap \Aut(F[W]; R)$.
	The image $\nu(\Aut(F[W], \mc U; R))$ is of finite index in $\nu(\Delta)$.
\end{lemma}
\begin{proof}
	Let $R' < R$ be a finite index characteristic subgroup of $F[W]$.
	Let $I(R') < \Inn(F[W])$ be the subgroup of conjugations by elements in $R'$. Since $R'$ is characteristic,
	$I(R')$ is normal in $\Aut(F[W])$. Since $I(R')$ is of finite index
	in $\Inn(F[W])$, we conclude that
	$\Aut(F[W], \mc U) \cdot I(R')$ is a finite index subgroup of $\Aut(F[W], \mc U) \cdot \Inn(F[W])$.
	
	Observe that $I(R') < \Delta$. Consequently, the following is a finite index subgroup
	of $\Delta$:
	$$ \begin{array}{rl} & (\Aut(F[W], \mc U)  \cdot I(R')) \cap \Delta \\
	 = & (\Aut(F[W], \mc U) \cap \Delta) \cdot I(R')\\
	 = &\Aut(F[W], \mc U; R) \cdot I(R'). \end{array} $$
	Moreover, $I(R')$ acts trivially on $R/[R, R]$, and hence lies in the kernel of $\rho$. Consequently,
	$\nu(\Aut(F[W], \mc U; R) \cdot I(R')) = \nu(\Aut(F[W], \mc U; R))$.
\end{proof}

The group $\Delta$ in the lemma is the image of a finite index subgroup of $\Aut^0(A_\Gamma)$. Thus,
any virtual quotient of $\Aut(F[W], \mc U)$  produced via $\nu$ is also a virtual quotient of $\Aut^0(A_\Gamma)$.

\subsection{Relating actions on homologies}
In the introduction, we briefly mentioned that our virtual arithmetic quotients arise from the action of
$\Aut(A_\Gamma)$ on the first homology of some finite index subgroup of $A_\Gamma$. For the sake of honesty, we 
show how virtual representations defined via $\Aut(A_\Gamma) \to \Aut(F[W], \mc U)$ and $\nu$
can be interpreted as being induced from an action of $\Aut(A_\Gamma)$ on some homology
of a finite index subgroup of $A_\Gamma$. Let $R' < A_\Gamma$ be the preimage of $R$ under
the quotient map $A_\Gamma \to F[W]$. Then $R = R'/\langle\langle \Gamma - W \rangle \rangle$
and similarly $\HH_1(R')/\langle\langle \Gamma - W \rangle \rangle = \HH_1(R)$. If we let
$$\Aut(A_\Gamma; R) = \{  \varphi \in \Aut(A_\Gamma) \;\; | \;\; q' \circ \varphi = q' \}$$
where $q' : A_\Gamma \to A_\Gamma/R'$ is the quotient map, then the action of
$\Aut(A_\Gamma; R)$ on $\HH_1(R')/\langle\langle \Gamma - W \rangle \rangle$ is equivalent
to the action obtained by first projecting to $\Aut(F[W], \mc U)$ and applying $\nu$.

\section{The virtual representations of $\Aut(F[W], \mc U)$} \label{section:definingrho}
Let $W$ be a finite set, and let $\preceq$ be a preorder on $W$. Let $\mc U$ and $\Aut(F[W], \mc U)$ be 
as in the previous section.
The virtual representations of $\Aut(F[W], \mc U)$ that we are interested in are both an extension
and projection of $\nu$ defined in the previous section. Also as in the previous section,
the representations can be defined virtually on $\Aut(F[W])$, not just the subgroup $\Aut(F[W], \mc U)$,
and so we will only restrict to $\Aut(F[W], \mc U)$ or $\Aut(F[W], \mc U; R)$ as it becomes necessary.

We will modify $\nu$ in
a few ways. First, we consider the action on rational homology, $\HH_1(R; \Q) = \HH_1(R) \otimes_\Z \Q$.
Second we extend this in a natural way to an action on cellular $1$-chains where we view $R$
as the fundamental group of the corresponding cover of a $K(F[W], 1)$. Finally, using the fact that the
action is by $\Q[G]$-module automorphisms where $G = F[W]/R$, we will restrict to some isotypic component.

\subsection{Defining the action on cellular $1$-chains} \label{section:hatY}
We identify $F[W] \cong \pi_1(Y, y)$ where $Y$ is a graph with $1$ vertex and one oriented
edge for each $w \in W$ which, by abuse of notation, will also be referred to as $w$.
Letting $\hat Y \to Y$ be the cover corresponding to a finite index normal subgroup $R < F[W]$,
we see that $\HH_1(\hat Y; \Q) = \HH_1(R; \Q)$. The action of $G$ on $\hat Y$ induces a left action
on $\HH_1(\hat Y; \Q)$.  \\

The cover $\hat Y$ inherits an orientation from $Y$. Choose a basepoint $\hat y$
in the $0$-skeleton of $\hat Y$. Let $\hat w$ denote that
oriented edge in $\hat Y$ which lies above the loop $w$ and is outgoing from $\hat y$. See
Figure \ref{figure:graphcover}.
Since the cover is regular, the $0$-skeleton of $\hat Y$ is precisely the orbit $G \cdot \hat y$,
and the lifts of loop $w$ are precisely the edges in the the $G$-orbit $G \cdot \hat w$.

\begin{figure}[h]
	\centering
	\labellist
	\pinlabel $u$ at 25 78
	\pinlabel $v$ at 108 78
	\pinlabel $w$ at 67 -5
	\pinlabel $y$ at 67 58
	\pinlabel $\hat y$ at -2 121
	\pinlabel $\hat u$ at -5 181
	\pinlabel $\hat w$ at 60 142
	\pinlabel $\hat v$ at 21 159
	\endlabellist
	\includegraphics[width= .3 \textwidth]{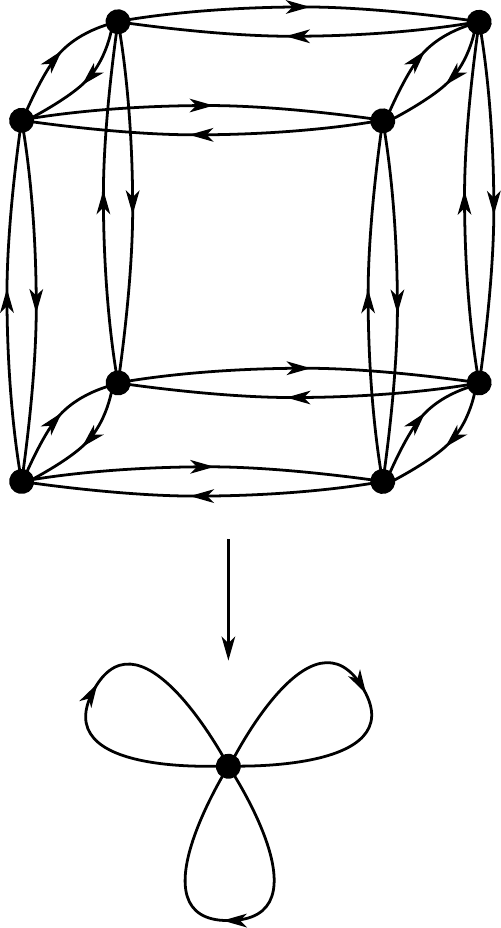}
	\caption{ A regular cover of $Y$ with $\hat u, \hat v, \hat w$ labeled. }
	\label{figure:graphcover}
\end{figure}
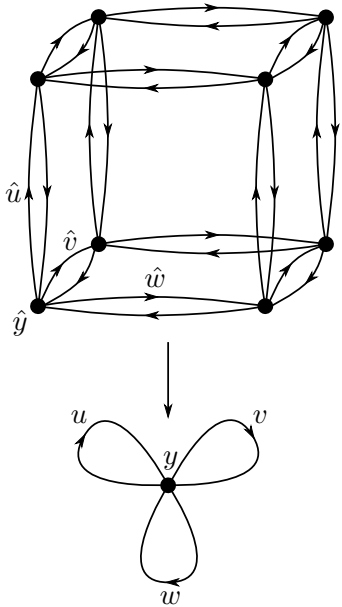

Following \cite{Hatcher}[Section 2.2],
let $C_1(\hat Y; \Q) = \HH_1(\hat Y^1, \hat Y^0; \Q)$ be the cellular $1$-chain group.
There is a well-defined $G$-equivariant action of $\Aut(F[W]; R)$
on cellular $1$-chains $C_1(\hat Y; \Q)$ defined as follows.
Any element $\varphi \in \Aut(F[W]; R)$ defines a homotopy equivalence $f\colon Y \to Y$ fixing $Y^0 \subset Y$
which lifts to a homotopy equivalence $\hat f\colon \hat Y \to \hat Y$ where $\hat f(\hat Y^0) = \hat f(\hat Y^0)$.
Since $\Aut(F[W]; R)$ acts trivially on $G = F[W]/R$, one can deduce that $\hat f$ commutes with the deck group $G$,
and so it acts as a $\Q[G]$-module automorphism on the $\Q[G]$-module $C_1(\hat Y; \Q)$.
Also, $\hat f_*\colon C_1(\hat Y, \Q) \to C_1(\hat Y; \Q)$ restricts to the action of $\varphi$ on 
$\HH_1(\hat Y; \Q) = \HH_1(R; \Q)$.
We let $\tilde{\nu}\colon \Aut(F[W]; R) \to \Aut_{\Q[G]}(C_1(\hat Y, \Q))$ be this action. 

We can also describe this action in the following way. Given $v \in W$, then $\varphi(v)$ is some reduced
word in $F[W]$ which defines a canonical loop in $Y$ whose lift defines a canonical path in $\hat Y$.
This in turn defines a canonical $1$-chain in $C_1(\hat Y, \Q)$ which we call $\overline{\varphi(v)}$.
Then, 
$\tilde{\nu}(\varphi)$ is the unique automorphism satisfying $\nu(\varphi)(\hat v) = \overline{\varphi(v)}$.

\subsection{Isotypic components} \label{section:isotypcomps}
Suppose $\Q[G] = B \times B'$ for some subrings $B, B'$ where $B$ 
is a simple $\Q$-algebra.
Moreover, let $1_B$ be the unit of $B$, and define
the $B$-isotypic component of any $\Q[G]$-module $M$ as $1_B \cdot M$. 
Let $C_1(\hat Y; \Q)_B$ (resp. $\HH_1(\hat Y; \Q)_B$) be the $B$-isotypic component of $C_1(\hat Y; \Q)$
(resp. $\HH_1(\hat Y; \Q)$). Note that any $\Q[G]$-module automorphism
preserves the $B$-isotypic component and acts as a $B$-module automorphism.
We can thus restrict the action to any isotypic component, i.e. to $\Aut_B(C_1(\hat Y; \Q)_B)$
or $\Aut_B(\HH_1(\hat Y; \Q)_B)$. Given a choice of isotypic component $B$, we
define 
$$\tilde{\rho}_B \colon \Aut(F[W]; R) \to \Aut_B(C_1(\hat Y; \Q)_B)$$ 
$$\rho_B \colon \Aut(F[W]; R) \to \Aut_B(\HH_1(\hat Y; \Q)_B)$$ 
as restrictions of $\tilde \nu$. Note that $\rho_B$ is also a restriction of $\nu$.

\begin{remark}
	It is well known that $\Q[G]$ fully decomposes as a product of simple $\Q$-algebras, one for each
	irreducible $\Q$-representation of $G$ up to $\Q$-isomorphism. In particular, there is always
	a factor $B \cong \Q$ where the projection $\Q[G] \to \Q$ maps $g \mapsto 1$ for all $g \in G$.
	We call this the trivial factor of $\Q[G]$.
\end{remark}

\subsection{Linear algebraic groups} \label{section:linalggrp}
We now define more refined subgroups of $\Aut_B(\HH_1(\hat Y; \Q)_B)$ and
of $\Aut_B(C_1(\hat Y; \Q)_B)$ which are equivalent to $\mc G$ and $\tilde{\mc G}$
defined in the introduction. We will assume throughout this section that $B$ is not the trivial
factor of $\Q[G]$. We begin first with $B$-automorphisms of $B$-modules and then proceed to
$\mcO$-modules where $\mcO$ will be an order of $B$.
First, note that
$C_1(\hat Y; \Q)$ is a free $\Q[G]$-module on $W$, with generators $\{\hat w \mid w \in W\}$. Consequently,
if we let $\hat w_B = 1_B \hat w$, then 
$\{\hat w_B \mid w \in W\}$ is a free generating set of $C_1(\hat Y; \Q)_B \cong B^n$.
Since there are no $2$-dimensional cells in $\hat Y$, we have that $\HH_1(\hat Y; \Q)$ is a subspace
of $C_1(\hat Y; \Q)$, and moreover it is a $\Q[G]$-submodule since it is invariant under $G$.
By a theorem of Gasch\"utz \cite{Gas, GLLM},$$\HH_1(\hat Y; \Q) \cong \Q[G]^{n-1} \oplus \Q,$$ and so
$\HH_1(\hat Y; \Q)_B \cong B^{n-1}$, since $B$ is not the trivial factor. In section \ref{section:lowerbound},
we will explicitly describe $\HH_1(\hat Y; \Q)_B$ as a submodule of $C_1(\hat Y; \Q)_B$.

In the definition of our subgroups, we use reduced norms, which we define now. 
Suppose $A$ is a finite-dimensional central simple $K$-algebra for some field $K$. 
Then, for some finite field extension $L \geq K$
and some integer $m$, there is an isomorphism $\psi\colon A \otimes_K L \cong \Mat_m(L)$. For any $a \in A$,
the reduced norm of $a$ (over $K$), which we will denote $\nrd(a)$, is the determinant of $\psi(a \otimes 1)$. The reduced norm is known
to lie in $K$ and to be independent of the choice of $L$ and the isomorphism $\psi$ \cite{Reiner}[Section 9]. Given a free $B$-module $B^k$, 
we have the canonical embedding $\Aut_B(B^k) \hookrightarrow \End_B(B^k)$.
Since $B$ is a finite-dimensional simple $\Q$-algebra, so is $\End_B(B^k)$, and, moreover, 
$\End_B(B^k)$ is a central simple algebra over its center which is a finite extension of $\Q$.
In this way, we define the reduced norm of an element of $\Aut_B(B^k)$ as that of the corresponding element in $\End_B(B^k)$.

For any subset $U \subseteq W$, we define
$C_B[U]$ to be the $B$-submodule of $C_1(\hat Y, \Q)_B$ generated by $\{ \hat u_B \;\; | \;\; u \in U\}$.
We denote the equivalence class of $w \in W$ under the preorder $\preceq$ by $[w]$.
We define $\tilde{\mc G}(B)$ to be the subgroup of $\Aut_B(C_1(\hat Y, \Q)_B)$ consisting
of automorphisms which 
\begin{enumerate}
	\item preserve the submodules $C_B[U(v)]$ for all $v \in W$,
	\item when restricted to $C_B[U(v)]$, have reduced norm $1$,
	\item act trivially on $C_B[U(v)]/C_B[U'(v)]$ for all $v$ with trivial equivalence classes,
	\item preserve $\HH_1(\hat Y; \Q)_B$,
	\item act trivially on $C_1(\hat Y, \Q)_B/\HH_1(\hat Y; \Q)_B$.
\end{enumerate}
We let $\mc G(B)$ to be the image of the restriction map $\Aut_B(C_1(\hat Y, \Q)_B) \to \Aut_B(\HH_1(\hat Y; \Q)_B)$.

Recall that an order $\mcO \subseteq B$ is a subring of $B$ which spans $B$ as a $\Q$-vector space
and is isomorphic to a free $\Z$-module. Although orders always exist in finite-dimensional simple $\Q$-algebras,
here we can explicitly take
$\mcO \subseteq B$ to be the image of $\Z[G]$ under the projection $\Q[G] \to B$. Then, the
subgroup of $\tilde{\mc G}(B)$ preserving the $\mcO$-span of $\{ \hat v_B \;\; |\;\; v \in W\}$ is
canonically isomorphic to the group $\tilde{\mc G}$ defined in the introduction (where
we take $H_\mcO$ as the intersection of $\HH_1(\hat Y; \Q)_B$ with the $\mcO$-span of the $\hat v_B$). Similarly,
$\mc G$ is canonically isomorphic to the subgroup of 
$\mc G(B)$ preserving the intersection of $\HH_1(\hat Y; \Q)_B$ and
the $\mcO$-span of $\{ \hat v_B\;\; |\;\; v \in W\}$. A finite index subgroup of $\mc G$
also preserves the $\mcO$-span of some free generating set of $\HH_1(\hat Y; \Q)_B \cong B^{n-1}$.

We break our main technical theorems into two pieces, one providing an upper bound on the image
of $\Aut(F[W], \mc U; R)$
and the other a lower bound. A priori, $\rho_B(\Aut(F[W], \mc U; R))$ lies in $\Aut_B(\HH_1(\hat Y; \Q)_B)$,
so we must prove it virtually lies in $\mc G$. Originally, we needed an extra assumption on the
finite group $G$ to verify condition 2 was satisfied in the definition of $\mc G$, but since the announcement
that $\Aut(F_m)$ has property (T) for $m \geq 5$, even this assumption is not strictly necessary \cite{KKN, KNO}.
Nevertheless, we show that the image lies virtually in $\mc G$ under 
any of three assumptions including that $\Aut(F_m)$ has property (T) for $m \geq 5$.
There are also cases of preorders $\preceq$ (all trivial equivalence classes) and $B$ (its center is $\Q$) where condition 2 is
virtually satisfied automatically.
We state the analogous result for $\tilde \rho_B$ and $\tilde{\mc G}$.

\begin{theorem} \label{thm:maintechthmupperbound}
	Let $W$ be a finite set with a preorder, and let $\mc U$ be as defined above. 
	Let $G = F[W]/R$ be finite and $B$ a nontrivial simple factor of $\Q[G]$. Let
	$\mcO$ be the image of $\Z[G]$ in $\Q[G]$ and $\tilde{\mc G} < \tilde{\mc G}(B)$ the subgroup
	preserving the $\mcO$-span of $\{\hat v_B \mid v \in W\}$.
	Let $\tilde{\rho}_B$ be the representation
	defined above. Then, $\tilde{\rho}_B(\Aut(F[W], \mc U; R))$
	virtually lies in $\tilde {\mc G}$ if any of the following hold.
	\begin{enumerate}
		\item  $G$ is a metabelian group.
		\item  $B \cong \Mat_m(\Q)$ for some $m \in \N$.
		\item  $\Aut(F_m)$ has property $(T)$ for all $m \geq 5$.
	\end{enumerate}
\end{theorem}

\begin{cor} \label{cor:maintechcorupperbound}
	Let notation and assumptions be as in Theorem \ref{thm:maintechthmupperbound}. Let 
	$\rho_B$ be the representation defined above and $\mc G$ the restriction of $\tilde{\mc G}$ to $\HH_1(\hat Y; \Q)_B$. 
	Then, $\rho_B(\Aut(F[W], \mc U; R))$ virtually lies in $\mc G$ if any of the following hold.
	\begin{enumerate}
		\item  $G$ is a metabelian group.
		\item  $B \cong \Mat_m(\Q)$ for some $m \in \N$.
		\item  $\Aut(F_m)$ has property (T) for all $m \geq 5$.
	\end{enumerate}
\end{cor}

\noindent In proving the lower bound, we need to assume some conditions on $G, B, q,$ and $\preceq$.
\begin{definition}
The conditions $\ddagger$ on $G, B, q, \preceq$ are as follows.
\begin{enumerate}
	\item If $U'(v) \neq \emptyset$, then there exists a maximal $u \in U'(v)$ such that $1_B(q(u) - 1)$ is invertible in $B$.
	\item If $v$ is nonmaximal and $|U'(v)| \geq 2$, then $q(U'(v))$ generates $G$. For $v$ maximal, $q(U'(v))$ generates $G$
	if $|U'(v)| = 2$ and $q(U(v))$ generates $G$ if $|U'(v)| \geq 3$.
	\item No maximal equivalence class is of size $2$.
	If there is a maximal equivalence class of size $3$ or a nonmaximal equivalence class of size $2$,
	then $B$ is not a field or a division algebra.
\end{enumerate}
\end{definition}

\begin{remark} Note that condition 1 does imply that maximal equivalence classes of size $\geq 3$ contain two
	vertices where $1_B(q(u) - 1)$ is invertible. Also, condition 1 implies that $B$ cannot be the trivial factor
	since, in that case, $1_B(q(u)-1) = 0$ for all $u \in W$.
\end{remark}

\begin{theorem} \label{thm:maintechthmlowerbound}
	Let notation and assumptions be as in Theorem \ref{thm:maintechthmupperbound}.  
	Suppose $G, B, q, \preceq$ satisfy $\ddagger$. Then $\tilde{\rho}_B(\Aut(F[W], \mc U; R))$ contains a finite index
	subgroup of $\tilde {\mc G}$.
\end{theorem}

\begin{cor} \label{cor:maintechcorlowerbound}
	Let notation and assumptions be as in Theorem \ref{thm:maintechthmupperbound} and Corollary \ref{cor:maintechcorupperbound}. 
	Suppose $G, B, q, \preceq$ satisfy $\ddagger$. Then $\rho_B(\Aut(F[W], \mc U; R))$ contains a finite index
	subgroup of $\mc G$.
\end{cor}

Using the above corollaries and a short argument, we can produce virtual arithmetic quotients
of the relative outer automorphism group.

\begin{cor} \label{cor:reloutquots}
	Let notation and assumptions be as in Theorem \ref{thm:maintechthmupperbound} and Corollary \ref{cor:maintechcorupperbound}. 
	Suppose $G, B, q, \preceq$ satisfy $\ddagger$. Then there is a virtual surjective representation
	$\Out(F[W], \mc U) \to \mc G$ induced by $\rho_B$.
\end{cor}

\begin{remark}
	The conditions $\dagger$ imposed on $W$ in the introduction may appear to have little effect on the above theorems. Indeed,
	most conditions from $\dagger$ have no impact here except to allow us to reduce to $\Aut(F[W], \mc U)$. 
\end{remark}

\noindent Corollary \ref{cor:reloutquots} and Lemma \ref{lemma:AutAGamtoAutFn} imply the following.

\begin{cor} \label{cor:generaloutagamquots}
	Let $\Gamma$ be a finite graph, $W \subseteq V(\Gamma)$ a subset satisfying $\dagger$, and $\preceq$ the restriction
	of the domination relation to $W$. Let $G = F[W]/R$ be finite, $B$ a nontrivial simple factor of $\Q[G]$, and
	$\mcO$ the image of $\Z[G]$ in $\Q[G]$. Let $\tilde{\mc G} < \tilde{\mc G}(B)$ be the subgroup
	preserving the $\mcO$-span of $\{\hat v_B \mid v \in W\}$ and $\mc G$ the restriction of $\tilde{\mc G}$
	to $\HH_1(\hat Y; \Q)_B$. Suppose $G, B, q, \preceq$ satisfy $\ddagger$.
	Then there is a virtual surjective representation $\Out(A_\Gamma) \to \mc G$.
\end{cor}

\section{Upper bound on the image of $\tilde{\rho}$} \label{section:upperbound}
Our goal is to show that $\tilde{\rho}_B(\Aut(F[W], \mc U; R))$ is virtually contained in $\tilde{\mc G}$.
Throughout this section, we continue to use the notation of Section \ref{section:definingrho}
and Theorem \ref{thm:maintechthmupperbound}. Unless otherwise indicated,
this is what the notation in the lemmas of this section refer to.
Moreover, we proceed to simplify some notation as follows. We let 
\begin{itemize}
\item $F = F[W]$
\item $\rho = \rho_B$
\item $\tilde{\rho} = \tilde{\rho}_B$
\item $C = C_1(\hat Y; \Q)$ (cellular $1$-chains) and $C_B = C_1(\hat Y; \Q)_B$
\item $H = H_1(\hat Y; \Q)$ and $H_B = H_1(\hat Y; \Q)_B$
\end{itemize}

We first show that the image virtually satisfies conditions 1, 3, 4, 5 of the definition of $\tilde{\mc G}$,
and so we let $\tilde{\mc G}'(B) < \tilde{\mc G}(B)$ be the subgroup satisfying only those conditions.
We let $\tilde{\mc G'} < \tilde{\mc G}'(B)$ be the subgroup preserving the $\mcO$-span of the $\hat v_B$.
In the definition of $\tilde{\mc G}(B)$, various subgroups of $C_B$ were required to be preserved, and
the actions on certain quotients were required to be trivial. We will in fact prove that these 
properties all analogously hold for the action of $\tilde{\nu}(\Aut(F[W], \mc U; R))$ on the entire chain group $C$. 

\begin{lemma} \label{lemma:actiononwords}
	Let $\varphi \in \Aut(F[W], \mc U; R)$. Then,
	\begin{itemize}
		\item $\varphi(v) \in F[U(v)]$ for all $v \in W$
		\item For all $v \in W$ that have a trivial equivalence class, there are words
		$x_1, x_2 \in F[U'(v)]$ such that $\varphi(v) = x_1 v^{\pm 1} x_2$.
	\end{itemize}
\end{lemma}
\begin{proof}
	The first statement is clear by definition of $\Aut(F[W], \mc U; R)$.
	If $v$ has a trivial equivalence class, then $F[U'(v)]$ is invariant
	under $\Aut(F[W], \mc U; R)$. The second statement then follows from the fact that, for all
	generators (or their inverses) in $\varphi \in \Aut(F[W], \mc U; R)$, there exists $u \in U'(v)$
	such that $\varphi(v) \in \{vu^{\pm 1}, u^{\pm 1}v, v^{\pm 1}\}$.
\end{proof}

\subsection{Acting on $1$-chains}
One helpful tool in describing the action on $1$-chains is the following
lemma relating words in $F$ and their $1$-chains in the cover $\hat Y$.
For any word $x \in F$, there is a corresponding based loop in $Y$ which lifts to a unique path in 
$\hat Y$ starting at $\hat y \in \hat Y$. Recall that $\overline{x}$ denotes the corresponding element of 
$C$.

\begin{lemma} \label{lemma:foxcalc}
	Suppose $x_1, x_2 \in F[W]$. Let $g_1 = q(x_1)$.
	Then, $\overline{x_1 x_2} = \overline{x_1} + g_1 \cdot \overline{x_2}$.
\end{lemma}
\begin{proof}
	Following the discussion before the lemma, let $\gamma_1, \gamma_2$ be the paths in $\hat Y$
	starting at $\hat y$ and corresponding to $x_1, x_2$ respectively. Let $\gamma_{12}$
	be the path in $\hat Y$ corresponding to $x_1 x_2$. Since the terminal endpoint of $\gamma_1$
	is $g_1 \cdot y$, uniqueness of path lifting tells us that the second half
	of the path $\gamma_{12}$ is simply $\gamma_2$ translated by $g_1$. The lemma follows.
\end{proof}

For a subset $U \subseteq W$, we now define $H[U] = C[U] \cap H$ and $H_B[U] = C_B[U] \cap H_B$;
i.e. these are the homology classes supported on $U$.

\begin{lemma} \label{lemma:finallyniceform}
	There is a finite index subgroup $\Delta < \Aut(F[W], \mc U; R)$ such that for all $\varphi \in \Delta$,
	the following hold.
	\begin{itemize}
		\item For all $v \in W$, we have $\tilde{\nu}(\varphi)(\hat v) = \hat v + z$ for some $z \in H[U(v)]$.
		\item If $v \in W$ has a trivial equivalence class, then 
		$\tilde{\nu}(\varphi)(\hat v) = \hat v + z $ for some $z \in H[U'(v)]$.
	\end{itemize}
\end{lemma}
\begin{proof}
	We begin with the first statement which holds for all $\varphi \in \Aut(F[W], \mc U; R)$.
	It follows from Lemma \ref{lemma:actiononwords} that $\varphi(v) \in F[U(v)]$, and consequently,
	by Lemma \ref{lemma:foxcalc}, $\tilde{\nu}(\varphi)(\hat v) \in C[U(v)]$.
	Moreover, since $\varphi \in \Aut(F[W], \mc U; R)$, we know that $r = v^{-1}\varphi(v) \in R$,
	and thus $$\tilde{\nu}(\varphi)(\hat v) = \overline{vr} = \hat v + q(v) \cdot \overline{r}.$$
	Note that $q(v) \overline{r} \in H[U(v)]$.
	
	To prove the second statement, we first prove a weaker claim for the entire group $\Aut(F[W], \mc U; R)$.
	Suppose $v \in W$ has a trivial equivalence class. 
	We claim that for all $\varphi \in \Aut(F[W], \mc U; R)$ there is some $g \in G$
	such that $\tilde{\nu}(\varphi)(\hat v) = \pm g \cdot \hat v + z$ and $z \in C[U'(v)]$. 
	By Lemma \ref{lemma:actiononwords}, 
	$\varphi(v) = x_1 v^{\pm 1} x_2$ for $x_i \in F[U'(v)]$.
	Let $g_1 = q(x_1)$ and $g_v = q(v)$. Note that $\overline{v^{-1}} = - g_v^{-1} \hat v$. Then, by Lemma \ref{lemma:foxcalc},
	$\tilde{\nu}(\varphi)(\hat v)$ is one of the following depending on the exponent of $v$.
	$$ \overline{x_1 v x_2}  = \overline{x_1} + g_1 \cdot \hat v + g_1 g_v \cdot \overline{x_2} $$
	$$ \overline{x_1 v^{-1} x_2}  = \overline{x_1} - g_1 g_v^{-1} \cdot \hat v + g_1 g_v^{-1} \cdot \overline{x_2} $$
	Clearly, $\overline{x_1} \in C[U'(v)]$ and $g_1 g_v^{\pm 1} \cdot \overline{x_2} \in C[U'(v)]$.
	The claim follows.
	
	Note that since $v$ has a trivial equivalence class, the set $U'(v)$ is a union
	of $U(u_1), \dots, U(u_t)$ for some $u_1, \dots, u_t \in U'(v)$. Consequently,
	$C[U'(v)] = \sum_i C[U(u_i)]$ is invariant under $\tilde{\nu}(\varphi)$ by the first
	part of the proof. Similarly $U(v)$ is invariant under $\tilde{\nu}(\varphi)$.
	Thus, there is a well-defined representation 
	$\Aut(F[W], \mc U; R) \to \Aut(C[U(v)]/C[U'(v)])$ by restricting and passing to the quotient.
	Moreover, by the previous paragraph, the representation has finite image. Thus, there is a finite index subgroup
	of $\Aut(F[W], \mc U; R)$ where, for this particular $v$, all its automorphisms $\varphi$ act by 
	$\tilde{\nu}(\varphi)(\hat v) = \hat v + z $ where $z \in C[U'(v)]$. Moreover, arguing as in the previous
	case, we must have $z \in C[U'(v)] \cap H$. By intersecting such subgroups
	over all such $v$, we obtain the desired $\Delta < \Aut(F[W], \mc U; R)$.
\end{proof}

\begin{cor} \label{cor:upperboundallofC}
	There is a finite index subgroup $\Delta < \Aut(F[W], \mc U; R)$ such that every automorphism of $\tilde{\nu}(\Delta)$
	\begin{itemize}
		\item preserves $C[U(v)]$ for all $v \in W$
		\item acts trivially on $C[U(v)]/C[U'(v)]$ for all $v$ with trivial equivalence classes
		\item preserves $H < C$
		\item and acts trivially on $C/H$
	\end{itemize}
\end{cor}
\begin{proof}
	This follows easily for $\Delta$ as in Lemma \ref{lemma:finallyniceform}.	
\end{proof}

\begin{cor} \label{cor:upperboundworednrm}
	There is a finite index subgroup $\Delta < \Aut(F[W], \mc U; R)$ such that $\tilde{\rho}(\Delta) < \tilde{\mc G}'$
\end{cor}

\begin{proof}
	It follows immediately that $\tilde{\rho}(\Delta) < \tilde{\mc G}'(B)$ for $\Delta$ as in Corollary \ref{cor:upperboundallofC}.
	Now, suppose $\varphi \in \Delta$. Since $\tilde{\nu}(\varphi)$ is the action of some continuous map
	fixing the $G$ orbit of the base vertex $\hat y \in \hat Y$, we have $\tilde{\nu}(\varphi)$ preserves
	$C_1(\hat Y; \Z) \subset C_1(\hat Y; \Q)$ or equivalently the $\Z[G]$-span of $\{\hat v \mid v \in W\}$. Since 
	we chose $\mcO \subseteq B$ to be the image of $\Z[G]$, we see that $\tilde{\rho}(\varphi)$
	preserves the $\mcO$-span of $\{\hat v_B \mid v \in W\}$.
\end{proof}

\subsection{Reduced norms}
We now show that $\tilde{\rho}(\Aut(F[W], \mc U; R))$ lies in $\tilde{\mc G}$ provided that $G$
has special properties or assuming $\Aut(F_n)$ has property (T) for $n \geq 5$.
We recall Proposition 8.7 from \cite{GL} but translated into our language.

\begin{prop} \label{prop:rednrmifmetabelian}
	Let $n \geq 3$ and $q: F_n \to G$ a surjective homomorphism with finite index kernel $R$.
	Let $B$ be a nontrivial simple factor of $\Q[G]$, and
	let $\rho_B: \Aut(F_n; R) \to \Aut(H_B)$ be the natural action on the $B$-isotypic component 
	of $\HH_1(R)$. If $G$ is a finite metabelian group, then
	there is a finite index subgroup $\Delta < \Aut(F_n; R)$ such that 
	$\rho(\varphi)$ has reduced norm $1$ for all $\varphi \in \Delta$.
\end{prop}

\noindent The goal of the section is the following lemma.

\begin{lemma} \label{lemma:rednrm1ifgoodGorpropT}
	Suppose that $G$ is metabelian, that $B = \Mat_m(\Q)$ for $m \geq 1$ or that $\Aut(F_m)$ has property (T) for all $m \geq 5$. 
	Then, $\tilde{\rho}(\Aut(F[W], \mc U; R))$ lies virtually in $\tilde{\mc G}(B)$. 
\end{lemma}

\noindent The proof will rely on the following two lemmas which we prove later.

\begin{lemma} \label{lemma:okayifmetab}
	Suppose $G$ is a finite metabelian group.
	Then, there is a finite index subgroup $\Delta < \Aut(F[W], \mc U; R)$ such that 
	$\tilde{\rho}_{B}(\varphi)$ has reduced norm $1$ for all $\varphi \in \Delta$.	
\end{lemma}

\begin{lemma} \label{lemma:rednrmifpropT}
	Suppose $\Aut(F_m)$ has property $(T)$ for $m \geq 5$.
	Then, there is a finite index subgroup $\Delta < \Aut(F[W], \mc U; R)$ such that 
	$\tilde{\rho}(\varphi)$ has reduced norm $1$ for all $\varphi \in \Delta$.
\end{lemma}

Our proof will involve the reduced norm of the restriction of automorphisms to invariant
submodules. Reduced norm behaves essentially as determinants in this respect, and the precise
statement is contained in Lemma \ref{lemma:nrdonparabolic} in the appendix.

\begin{proof}[Proof of Lemma \ref{lemma:rednrm1ifgoodGorpropT}]
	We first consider the case where $B \cong \Mat_m(\Q)$ for some $m \geq 1$. In this case, 
	$\End_B(B^n) \cong \Mat_n(B^{op}) \cong \Mat_{mn}(\Q)$ is already isomorphic to a matrix
	algebra over a field, and so reduced norm is just the determinant. Up to changing $\tilde{\mc G}$
	to a commensurable group, we may assume $\mcO = \Mat_m(\Z)$, and so
	determinants of elements of $\tilde{\mc G}'$ (and restrictions to subspaces) lie in $\Z$. One can then readily deduce
	that $\tilde{\mc G}'$ is of finite index in $\tilde{\mc G}$.
	
	We now consider the cases where $G$ is metabelian or $\Aut(F_m)$ has property (T) for $m \geq 5$. 
	Fix arbitrary $v \in W$ and let $U = U(v)$. It suffices to show that
	all reduced norms are $1$ when restricting to $C_B[U]$ after 
	possibly passing to a finite index subgroup of $\Aut(F[W], \mc U; R)$.
	Since $F[U] < F[W]$ is invariant under this group, we can define a homomorphism
	$s \colon \Aut(F[W], \mc U; R) \to \Aut(F[W], \mc U; R)$ by $s(\varphi)|_{F[U]} = \varphi|_{F[U]}$
	and $s(\varphi)|_{F[W-U]} = \id$. Then, $\tilde{\rho}_B(s(\varphi))$ and 
	$\tilde{\rho}_B(\varphi)$ have identical actions on $C_B[U]$, but $\tilde{\rho}_B(s(\varphi))$
	is the identity on $C_B[W-U]$.
	
	Similarly, we can define $s' \colon \tilde{\mc G}'(B) \to \tilde{\mc G}'(B)$ by
	$s'(\psi)|_{C_B[U]} = \psi|_{C_B[U]}$ and $s'(\psi)|_{C_B[W-U]} = \id$.
	Clearly, $\tilde{\rho}_B(s(\varphi)) = s'(\tilde{\rho}_B(\varphi))$, and so
	$s'$ restricts to a homomorphism from $\tilde{\rho}_B(\Aut(F[W], \mc U; R))$
	to itself. By Lemmas \ref{lemma:okayifmetab} and \ref{lemma:rednrmifpropT},
	there is a finite index subgroup $\Delta$ of $\tilde{\rho}_B(\Aut(F[W], \mc U; R))$
	where all elements have reduced norm $1$. However, the reduced norm of $s'(\tilde{\rho}_B(\varphi))$
	is the same as the reduced norm of the restriction of $\tilde{\rho}_B(\varphi)$ to 
	$C_B[U]$ by Lemma \ref{lemma:nrdonparabolic}. The preimage of $\Delta$ under $s'$ is the desired finite index subgroup of 
	$\tilde{\rho}_B(\Aut(F[W], \mc U; R))$.
\end{proof}

\begin{proof}[Proof of Lemma \ref{lemma:okayifmetab}]
	First, suppose $|W| \geq 3$.
	Let $\Delta' < \Aut(F[W]; R)$ be the subgroup from Proposition \ref{prop:rednrmifmetabelian}.
	Then, $\rho_B(\varphi)$ has reduced norm $1$ for all $\varphi \in \Delta$ where
	$\Delta = \Delta' \cap \Aut(F[W], \mc U; R)$. Since $\rho_B(\varphi)$ is
	the restriction of $\tilde{\rho}_B(\varphi)$ to $H_B$, and $\tilde{\rho}_B(\varphi)$
	is the identity on $C_B/H_B$, the reduced norm of $\tilde{\rho}(\varphi)$ is $1$
	for all $\varphi \in \Delta$ by Lemma \ref{lemma:nrdonparabolic}.
	
	Now, suppose $|W| \leq 2$, and extend $W$ to a superset $W'$ of size at least $3$.
	Extend the preorder $\preceq$ to $W'$ in some way and let 
	$$\mc U' = \{F[U] \mid U \subseteq W' \text{ is closed under upper bounds}\}.$$
	Let $R'$ be the kernel of the map $q': F[W'] \to G$ which extends
	$q: F[W] \to R$ by mapping $W' - W$ to the identity. Let $Y', \hat Y'$ be the corresponding
	spaces for $W'$ as defined in Section \ref{section:hatY}. Embed $Y$ in $Y'$ in the obvious way.
	Since the deck group of $\hat Y' \to Y'$ has deck group $G$, the preimage of $Y$ in
	$\hat Y'$ is a copy of the cover $\hat Y$. As $C_1(\hat Y'; Q)$ is a free $\Q[G]$-module
	on the $\hat w'$, we have $C_1(\hat Y; Q)$ (resp. $C_1(\hat Y; \Q)_B$) embeds
	as a free $\Q[G]$-summand (resp. free $B$-summand) of $C_1(\hat Y'; Q)$
	(resp. $C_1(\hat Y'; Q)_B$). Letting 
	$\tilde{\rho}_B' \colon \Aut(F[W']; R) \to \Aut_B(C_1(\hat Y'; Q)_B)$
	be the representation defined as in Section \ref{section:definingrho},
	we obtain the following commutative diagram where the horizontal maps
	are extension by identity on the extra generators.
	\begin{center}
		\begin{tikzcd}
			\Aut(F[W], \mc U; R)  \arrow{r} \arrow{d}{\tilde{\rho}_B} & \Aut(F[W'], \mc U'; R') \arrow{d}{\tilde{\rho}_B'} \\
			\Aut_B(C_1(\hat Y; \Q)_B)  \arrow{r} & \Aut_B(C_1(\hat Y'; \Q)_B)
		\end{tikzcd}
	\end{center}
	Since extending by the identity on a summand doesn't change the reduced norm
	by Lemma \ref{lemma:nrdonparabolic},
	we are finished by the previous case applied to $\Aut(F[W'], \mc U'; R')$.
\end{proof}

\begin{proof}[Proof of Lemma \ref{lemma:rednrmifpropT}]
	First, suppose $|W| \geq 5$.
	Since $\Aut(F[W]; R)$ is a finite index subgroup of $\Aut(F[W])$, it maps to a finitely
	generated abelian group under $\nrd \circ \tilde{\rho}_B$. Since $\Aut(F[W])$ has property (T),
	this image must be finite. Thus, some finite index subgroup of $\Aut(F[W]; R)$ and thus
	of $\Aut(F[W], \mc U; R)$ has image where all elements have reduced norm $1$.
	If $|W| \leq 4$, we can reduce to the case $|W| \geq 5$ in the same manner as in the previous proof.
\end{proof}

\begin{proof}[Proof of Theorem \ref{thm:maintechthmupperbound}]
	Combine Corollary \ref{cor:upperboundworednrm} and Lemma \ref{lemma:rednrm1ifgoodGorpropT}.
\end{proof}

\begin{proof}[Proof of Corollary \ref{cor:maintechcorupperbound}]
	Apply Theorem \ref{thm:maintechthmupperbound}.
\end{proof}

\section{Lower bound on the image} \label{section:lowerbound}

It remains to show that $\tilde{\rho}(\Aut(F[W], \mc U; R))$ virtually contains $\tilde{\mc G}$. Similar to \cite{GLLM, GL}, 
we will do this by showing that the image contains sufficiently many unipotents.
We continue to use the same simplified notation as in Section \ref{section:upperbound}. 

\subsection{Generating $\SL$}
Our proof relies fundamentally on the theorem that $\SL_n(\mcO)$ for $n \geq 3$ is generated, up to finite index,
by elementary matrices. By $\SL_n(\mcO)$ or $\SL_n(B)$, we refer to the matrices in
$\Mat_n(\mcO)$ and $\Mat_n(B)$ of reduced norm $1$.
The precise statement below
is almost identical to Proposition 5.1 of \cite{GL} in content and follows from the main
result of \cite{Vas}.


\begin{prop} \label{prop:transvgen}
	Let $B = \Mat_m(E)$ for some finite dimensional division algebra $E$ over $\Q$. Let $n \geq 2$
	and assume $m \geq 2$ if $n = 2$. Let $\mcO \subseteq B$ be an order in $B$, and let
	$\mathfrak G < \Aut_{\mcO}(\mcO^n)$ be the subgroup of reduced norm $1$. Let
	$e_1, \dots, e_n$ be the canonical free generating set for the left $\mcO$-module $\mcO^n$.
	For any two-sided ideal $\mathfrak a \subseteq \mcO$, let 
	$\mathfrak T_{\mathfrak a} = 
		\{ T_{i,j,a} \in \mathfrak G \;\; | \;\; 1 \leq i, j \leq n, i \neq j, a \in \mathfrak a\}$
	where $T_{i,j,a}$ is defined by
	$$ T_{i,j,a}(e_i) = e_i + a e_j  \;\;\;\;\; T_{i,j,a}(e_k) = e_k \text{ if } k \neq i$$
	Then, the set $\mathfrak T_{\mathfrak a}$ generates a finite index subgroup of $\mathfrak G$.
\end{prop}
\begin{proof}
	In the case of $n = 3$, this follows directly from Proposition 5.1 of \cite{GL},
	and the canonical isomorphism $\mathfrak G \cong \SL_n(\mcO)^{op}$.
	We therefore consider the case $n = 2$, in which case $m \geq 2$ by assumption.
	Let $\mcO_E$ be an order of $E$. Then $\Mat_m(\mcO_E)$ is also an order of $B$, and since
	$B$ is a finite-dimensional simple $\Q$-algebra, $\mcO$ and $\Mat_m(\mcO_E)$ are commensurable.
	Consequently, $\mathfrak G$ and $\SL_{nm}(\mcO_E)^{op}$ are commensurable under the
	canonical identifications to subgroups of $\SL_2(B)^{op} \cong \SL_{2m}(E)^{op}$.
	Moreover, $\mathfrak a$ contains $\Mat_m(N\mcO_E)$ for some integer $N > 0$.
	
	Thus, $\mathfrak T_{\mathfrak a}$ contains the subgroups $H(2m, 2m; N\mcO_E)$
	and $V(2m, 2m; N\mcO_E)$ as defined in Proposition 5.1 of \cite{GL}.
	By that proposition, $H(2m, 2m; N\mcO_E)$ and $V(2m, 2m; N\mcO_E)$ generate
	a finite index subgroup of $\SL_{nm}(\mcO_E)$, and hence of
	$\SL_{nm}(\mcO_E)^{op}$. Thus, $\mathfrak T_{\mathfrak a}$ generates a
	finite index subgroup of $\mathfrak G$.
\end{proof}

\subsection{Preliminary Lemmas}
We begin with a few preliminary lemmas to help break down the proof of our theorems. We note that the lemmas
in this section depend on the definitions and notation of the theorems but don't require the assumptions $\ddagger$.
The proof of our theorems will require the use of automorphisms which map $v \mapsto zv$ for some $z \in F[U'(v)]$
and fix $w$ for $w \neq v$. Consequently, we start with a few lemmas describing the cycles $\overline{z}$
for $z$ supported on some free factor $F[U]$. Let $G(U)$ be the subgroup of $G$ generated by $q(U)$. 
It is clear that if $z \in F[U]$, then $\displaystyle \overline{z} = \sum_{v \in U} \alpha_v \hat v$
for some coefficients $\alpha_v \in \Q[G]$. Moreover, $\alpha_v$ must lie in $\Q[G(U)]$. We
provide partial converses of this statement.

To do this, we need to provide a description of the subspaces $H \subseteq C$ and $H_B \subseteq C_B$.
We temporarily revert to the expanded notation $H = \HH_1(\hat Y; \Q)$ and $C = C_1(\hat Y; \Q)$.
Recall that we view $\HH_1(\hat Y; \Q)$
as a cellular homology group, and that $\HH_1(\hat Y; \Q)$ is a 
$\Q[G]$-submodule of $C_1(\hat Y; \Q)$. Using $d_1: C_1(\hat Y; \Q) \to C_0(\hat Y; \Q)$ to denote the cellular
boundary map, we see that, for a $1$-chain $\hat v$, the image is $d_1(\hat v) = (q(v) - 1) \cdot \hat y$.
The cellular boundary maps are $G$-equivariant, and so we can restrict $d_1$
to $d_{1, B}\colon C_1(\hat Y; \Q)_B \to C_0(\hat Y; \Q)_B$, and then $d_{1, B}(\hat v_B) = 1_B(q(v)) - 1) \hat y$. 
Since $\HH_0(\hat Y; \Q) \cong \Q$ is the trivial $\Q[G]$-module, 
$d_{1, B}$ is surjective and $\HH_1(\hat Y; \Q)_B \cong B^{n-1}$ if $B$ is not the trivial
factor. Since they appear frequently, we define $a_v = q(v) - 1$ and $b_v = 1_B a_v$. The following
lemma is immediate.

\begin{lemma} \label{lemma:formofcycle}
	The $1$-chain $\displaystyle \sum_{v \in W} \alpha_v \hat v \in C$ is a $1$-cycle if and only if
	$\displaystyle \sum_{v \in W} \alpha_v a_v = 0$. The $1$-chain 
	$\displaystyle \sum_{v \in W} \beta_v \hat v_B \in C_B$ is a $1$-cycle if and only if
	$\displaystyle \sum_{v \in W} \beta_v b_v = 0$. 
\end{lemma}

\begin{lemma} \label{lemma:repcyclesbywords}
	Let $U \subseteq W$, and let $\displaystyle z = \sum_{v \in U} \alpha_v \hat v \in H[U]$.
	If $\alpha_v \in \Q[G(U)]$ for all $v \in U$, then there exists $M \in \Z$ such that
	$M z = \overline{x}$ for some word $x \in F[U] \cap R$.
\end{lemma}
\begin{proof}
	Let $Y' \subseteq Y$ be the subcomplex consisting
	of the single $0$-cell and the edges corresponding to $U$. Let $\hat Y' \subseteq \hat Y$
	be the component of the preimage of $Y'$ under the covering map $\hat Y \to Y$ such that 
	$\hat Y'$ contains the basepoint $\hat y$. 
	Then $\hat Y' \to Y'$ is precisely the covering space with deck group $G(U)$, and so the
	left action of $G(U)$ preserves $\hat Y'$. Since $z$ is a rational cellular $1$-cycle,
	some multiple $M z$ is an integral cellular $1$-cycle. Since $\alpha_v \in \Q[G(U)]$ and 
	the $1$-chain $\hat v$ is supported on $\hat Y'$, the $1$-cycle $M z$ lies entirely in $\hat Y'$.
	Consequently, $M z$ is homologous to some based loop in $\hat Y'$, and the lemma follows.
\end{proof}

We would like to say that a cycle satisfying the above in the $B$-isotypic component is the projection
of some cycle $\overline{x}$ for some $x \in F[U] \cap R$. While that's true, it is not immediate as the proof of the next lemma 
shows. We let $p: \Q[G] \to B$ denote the projection. Note that $1_B \alpha = p(\alpha)$ for $\alpha \in \Q[G]$.

\begin{lemma} \label{lemma:lifttogroupring}
	Let $U \subseteq W$, and let $\displaystyle z = \sum_{v \in U} \beta_v \hat v \in H_B[U]$.
	If $\beta_v \in p(\Q[G(U)])$ for all $v \in U$, then there exists $M \in \Z$ such that
	$M z = 1_B \overline{x}$ for some word $x \in F[U] \cap R$.	
\end{lemma}
\begin{proof}
	We can apply Lemma \ref{lemma:repcyclesbywords} if we
	can show that $z = 1_B \tilde z$ for some $\tilde z = \displaystyle \sum_{v \in U} \alpha_v \hat v \in H[U]$ 
	where $\alpha_v \in \Q[G(U)]$.
	However, if we pick arbitrary preimages $\alpha_v \in p^{-1}(\beta_v)$, there is no guarantee that we have a cycle.
	Since $z$ is a cycle, we know that $\displaystyle \sum_{v \in U} \beta_v b_v = 0$. It suffices
	to find $\alpha_v \in \Q[G(U)]$ for $v \in U$ satisfying $p(\alpha_v) = \beta_v$ and
	$\displaystyle \sum_{v \in U} \alpha_v a_v = 0$.
	
	Since $G(U)$ is finite, we know that $\Q[G(U)]$ is a semisimple algebra over $\Q$
	and thus it is a product $\Q[G(U)] \cong \prod_i A_i$
	for some finite collection $A_i$ of simple $\Q$-algebras. Thus, the image $p(\Q[G(U)])$
	is isomorphic to some subproduct $\prod_{i \in I} A_i$ and $p|_{\Q[G(U)]}$ is equivalent to the projection
	onto these factors. In other words, we may view $\beta_v \in \prod_{i \in I} A_i$, and we then
	choose $\alpha_v$ to be $\beta_v$ in the factor $\prod_{i \in I} A_i$ and $0$
	in $\prod_{i \notin I} A_i$.
	
	We then need to verify $\displaystyle \sum_{v \in U} \alpha_v a_v = 0$. Since $\alpha_v, a_v \in \Q[G(U)]$
	for all $v \in U$, it suffices to check equality after projecting to each factor $A_i$. If we project
	to the subproduct $\prod_{i \in I} A_i$, this is true since $\displaystyle \sum_{v \in U} \beta_v b_v = 0$.
	For $i \notin I$, all $\alpha_v$ project to $0$ and the equality is trivial.
\end{proof}

\begin{lemma} \label{lemma:invinonevar}
	Suppose $b_v$ is invertible in $B$. Then, $b_v^{-1} \in p(\Q[G(v)])$.
\end{lemma}
\begin{proof}
	As mentioned in the previous lemma, $\Q[G(v)]$ is a semisimple algebra. Moreover, since $G(v)$ is abelian,
	$\Q[G(v)]$ is a product of fields, $\prod_i K_i$. As above, $p|_{\Q[G(v)]}$ 
	is equivalent to projection onto some subproduct $\prod_{i \in I} K_i$. We view $b_v \in \prod_{i \in I} K_i$,
	and since it's invertible in $B$, it's neither a zero divisor in $B$ nor in $\prod_{i \in I} K_i \subset B$.
	Consequently, the components of $b_v$ in $\prod_{i \in I} K_i$ are nonzero, and 
	$b_v$ has an inverse in $\prod_{i \in I} K_i = p(\Q[G(v)])$.
\end{proof}

Our next lemma gives a slight refinement of the definition of $\tilde{\mc G}$. The definition of 
$\tilde{\mc G}$ requires that automorphisms have reduced norm $1$ when restricted to $C_B[U(v)]$.
In the next lemma, we show that elements of $\tilde{\mc G}(B)$ have reduced norm $1$ when passing to quotients of invariant subspaces.
Note that it follows from the definition of $\tilde{\mc G}(B)$ that, for any $U \subseteq W$ where $U$ contains all its upper bounds,
the subspace $C_B[U]$ is invariant under $\tilde{\mc G}(B)$. In particular, the subspace $C_B[U(W')-W']$ is invariant for any
equivalence class $W' \subseteq W$. For an equivalence class $W'$, we let 
$$\pi_{W'}: \tilde{\mc G}(B) \to \Aut_B( C_B[U(W')]/C_B[U(W')-W')]$$ be the induced action.

\begin{lemma} \label{lemma:rednrmonsubquot}
	Suppose $\psi \in \tilde{\mc G}(B)$. Then for any equivalence class $W' \subseteq W$, the induced automorphism
	$\pi_{W'}(\psi)$ has reduced norm $1$.
\end{lemma}
\begin{proof}
	Let $m(W')$ be the maximum length of a chain of equivalence classes with $W'$ as a lower bound. 
	We induct on $m(W')$. If $m(W') = 1$, then $W'$ is maximal, $C_B[U(W')-W'] = 0$ and the claim follows by definition
	of $\tilde{\mc G}(B)$.
	Now consider $m(W') > 1$. Then $U(W')$ is a union of equivalence classes $W_1, W_2, \dots, W_t = W'$.
	We can,  w.l.o.g., order the $W_i$ such that $W' = W_t$ and $W_i \preceq W_j$ only if $i \geq j$.
	Letting $U_0 = \emptyset$ and $\displaystyle U_i = \bigcup_{j=1}^i W_i$ for $i > 0$, we see that $U_i$ contains its upper bounds.
	In particular, the chain of free $B$-modules $C_B[U_0] \subseteq C_B[U_1] \subseteq  \dots \subseteq C_B[U_t]$ is invariant under $\psi$.
	Note that there is a natural isomorphism
	$$C_B[U_i]/C_B[U_{i-1}] \cong C_B[U(W_i)]/C_B[U(W_i)-W_i]$$ because $U(W_i)-W_i \subseteq U_{i-1}$.
	We thus have $1 = \nrd(\psi|_{C_B[U(W')]}) = \prod_{i=1}^t \nrd(\pi_{W_i}(\psi))$ by Lemma \ref{lemma:nrdonparabolic}. 
	Since $m(W_i) < m(W')$ for $i < t$, we conclude
	by induction that $\prod_{i=1}^{t-1} \nrd(\pi_{W_i}(\psi)) = 1$, and thus, $\nrd(\pi_{W'}(\psi)) = 1$. 
\end{proof}

\subsection{Proof of Theorem \ref{thm:maintechthmlowerbound}}
The proof of the theorem is a bit lengthy and is broken into three stages which are themselves broken down
into smaller pieces. The first stage will simply establish that the image has certain unipotent elements.
Specifically, we want to show that for any $v \in W$ and $z \in H_B[U'(v)]$,
that the image has the automorphism which adds $Mz$ to $\hat v_B$ and fixes $\hat w_B$ for $w \neq v$ 
for some integer $M > 0$. It is easy to do this if the coefficients of $z$, 
expressed as a sum of $\{\hat u_B \mid u \in U'(v)\}$, lie in $\Q[G(U'(v))]$ using the above lemmas.
However, we need to remove this restriction on coefficients in order to prove our theorem. 

For the second and third stages, we break $\tilde{\mc G}$ into two parts, one which is (mostly) a product
of $\SL$'s and one which is nilpotent. Specifically, we let $\pi$ be the product of representations
$\prod_{W'} \pi_{W'}$ on $\tilde{\mc G}$ where $W'$ ranges over all equivalence classes in $W$.
Let $\mcS$ be the image of $\tilde{\mc G}$ under $\pi$ and $\mcN$ the kernel. Let 
$\Imrho = \rho(\Aut(F[W], \mc U; R))$. We show 
in the second stage that $\Imrho$ contains a finite index subgroup of $\mcN$.
In the third stage, we show that $\pi(\Imrho \cap \tilde{\mc G})$ contains a finite index subgroup
of $\mcS$. Together, these show $\Imrho$ contains a finite index subgroup of $\tilde{\mc G}$.

\subsubsection{Stage 1: Unipotents in the image}

For distinct $u, w \in U'(v)$ and $b \in B$ where $b_w \in B$ is invertible, let
$$ e_{u, w} = \hat u_B - b_u b_w^{-1} \hat w_B. $$
Observe that by Lemma \ref{lemma:formofcycle}, $e_{u, w} \in H_B$.
For $v \in W$ and $z \in H_B[U'(v)]$, let $T_{v, z}$ denote the automorphism of $C_B$ satisfying
$$T_{v, z}(\hat v_B) = \hat v_B + z$$ 
$$T_{v, z}(\hat w_B) = \hat w_B \text{ for all } w \neq v.$$
Note that $T_{v, z} \in \tilde{\mc G}(B)$.


\begin{claim} \label{claim:gettransvforast}
	Suppose $v \in W$ satisfies $|U(v)| \geq 3$.
	Let $\mu \in U(v)$ be a maximal element such that $b_\mu$ is invertible, and let $u \in U'(v) - \mu$.  
	For any such $v, u, \mu$ and for all $b \in B$, there exists $M \in \Z$ such that
	$T_{v, z} \in \Imrho$ where $z = M b e_{u, \mu}$.
\end{claim}
\begin{remark} Note that such a $\mu$ exists by condition 1 of $\ddagger$.
\end{remark}
\begin{proof}
	Since $G$ is finite, it suffices to show that 
	$T_{v, z} \in \Imrho$ when $b = M p(g)$ for arbitrary $g \in G$ and appropriate integer $M > 0$ possibly depending on $g$.
	We first show that this holds for $g \in G(U'(v))$.
	By Lemmas \ref{lemma:lifttogroupring} and \ref{lemma:invinonevar}, there exists
	$x \in F[U'(v)] \cap R$ and $M > 0$ such that $1_B \overline{x} = z$.
	Let $\psi \in \Aut(F[W], \mc U; R)$ be the automorphism
	mapping $v \mapsto xv$ and $w \mapsto w$ for $w \neq v, w \in W$.
	(Note that there is such an automorphism precisely because $x \in F[U'(v)] \cap R$ and $v \notin U'(v)$.) 
	Then, $\rho(\psi) = T_{v, z}$.
	
	If $G(U'(v)) = G$, we are done.
	Assume now that $G(U'(v)) \neq G$. Condition 2 of $\ddagger$ then implies $v$ is maximal and $|U(v)| \geq 4$.
	Choose $w \in U(v)$ distinct from $u, v, \mu$.
	
	Now let $G_v = G(U'(v)), G_u = G(U'(u)),$ and $G_w = G(U'(w))$. Let $z_u = e_{u, \mu}, z_v = e_{v, \mu},$
	and $z_w = e_{w, \mu}$.
	 Since $G$ is finite, there is some sufficiently
	large $M$ such that $T_{v', be_{v''}} \in \Imrho$ for all distinct pairs $v', v'' \in \{w, u, v\}$ and all $b \in M p(G(U'(v')))$.
	We compute
	$$[T_{u, b' z_w}, T_{v, b z_u}] = T_{u, b' z_w}^{-1} T_{v, b z_u}^{-1} T_{u, b' z_w} T_{v, b z_u}
	 = T_{v, bb' z_w}$$
	 and similarly $$[T_{w, b''z_u}, T_{v, bb' z_w}] = T_{v, b b' b'' z_u}$$
	 In particular, we have $T_{v, b z_u} \in \Imrho$ for all $b \in M^3 p(G_v \cdot G_u \cdot G_w)$.
	 Taking a similar sequence of commutators, we find that 
	 $T_{v, b z_u} \in \Imrho$ for all $b \in M^5 p(G_v \cdot (G_u \cdot G_w)^2)$.
	 By induction, $T_{v, b z_u} \in \Imrho$ for all $b \in M^{2m+1} p(G_v \cdot (G_u \cdot G_w)^m)$.
	 Since the subgroups $G_u$ and $G_w$ generate $G$ by condition 2 of $\ddagger$, we find that there is some large $m$, where
	 $T_{v, b z_u} \in \Imrho$ for all $b \in M^{2m+1} p(G)$. This establishes the claim.
\end{proof}

\begin{claim} \label{claim:transvbyanycycle}
	For all $v \in W$ and $z \in H_B[U'(v)]$, there exists $M \in \Z$ such that
	$T_{v, Mz} \in \Imrho$.
\end{claim}
\begin{proof}
	We first consider the case where $|U'(v)| \leq 1$. In the case where $U'(v) = \emptyset$, this is a vacuous statement.
	In the case where $U'(v) = \{u\}$, condition 1 of $\ddagger$ requires that $b_u$ be invertible. 
	Thus, $H_B[U'(v)] = 0$ by Lemma \ref{lemma:formofcycle}. We now assume that $|U'(v)| \geq 2$.
	
	By conditions $\ddagger$, there is some maximal $\mu \in U'(v)$ such that $b_\mu$ is invertible. By Lemma \ref{lemma:formofcycle}
	$$z = \beta_\mu \hat \mu_B + \sum_{u \in U'(v) - \mu } \beta_u \hat u_B$$
	for some $\beta_\mu, \beta_u$ satisfying 
	$$\beta_\mu b_\mu + \sum_{u \in U'(v) - \mu} \beta_u b_u = 0.$$ 
	Since $b_\mu$ is invertible, we can rewrite this as
	$$\beta_\mu = - \sum_{u \in U'(v) - \mu} \beta_u b_u b_\mu^{-1}.$$ Consequently, $H_B[U'(v)]$ is generated as a $B$-module
	by the set $$\{ e_{u, \mu} \mid u \in U'(v) - \mu\}.$$ It therefore suffices to prove the claim
	for $z = be_{u, \mu}$ for arbitrary $b \in B$ and $u \in U'(v) - \mu$, but this follows
	from Claim \ref{claim:gettransvforast}.
\end{proof}

\begin{remark} Note that for claims \ref{claim:gettransvforast} and \ref{claim:transvbyanycycle}, we may further increase $M$ as necessary
	to ensure that $T_{v, Mz} \in \tilde{\mc G}$.
\end{remark}

\subsubsection{Stage 2: The nilpotent part}
Let $C_\mcO$ be the $\mcO$-span of $\{\hat v_B  \mid v \in W\}$. We will show that $\Imrho$ contains
the $M\mcO$-congruence subgroup of $\mcN$ for a sufficiently large integer $M$, i.e. the kernel of $\mcN \to \Aut_\mcO(C_\mcO/M C_\mcO)$ which
is finite index in $\mcN$. For a subset $U \subseteq W$, let $H_{\mcO}[U] = H_B[U] \cap C_\mcO$.

Since $H_\mcO$ is finitely generated, there exists a sufficiently large integer $M > 0$ such that 
$T_{v, z} \in \Imrho \cap \mcN$ for all $z \in M\cdot  H_{\mcO}[U(v)-[v]]$ by Claim
\ref{claim:transvbyanycycle}. 
Let $\psi$ be an element of the $M\mcO$-congruence subgroup of $\mcN$.
Using the definition of $\tilde{\mc G}$ and $\mcN$, we see, for all $v \in W$, that
$\psi(\hat v_B) - \hat v_B \in M \cdot H_{\mcO}[U(v)-[v]]$. (Moreover, any such automorphism lies in $\mc N$ by definition
of $\mc N$ and $\tilde{ \mc G}$.)
Order the vertices of $W$ as $v_1, v_2, \dots v_n$ such that 
$i \leq j$ implies $v_i \succeq v_j$ or $v_i, v_j$ are incomparable. Let 
$$z_i = \psi(\hat v_{iB}) - \hat v_{iB} \in M \cdot  H_{\mcO}[U(v)-[v]].$$ Then
$\psi = T_{v_n, z_n} \circ T_{v_{n-1}, z_{n-1}} \circ \dots \circ T_{v_1, z_1} \in \Imrho$. 

\subsubsection{Stage 3: The (mostly) SL part}
We want to show that $\pi(\Imrho \cap \tilde{\mc G})$  has finite index in 
$\mc S = \pi(\tilde{\mc G})$. Recall that $\pi = \prod_{W'} \pi_{W'}$. It will suffice to show for each equivalence
class $W'$, there is a finite index subgroup of $\pi_{W'}(\tilde{\mc G})$ generated by elements
of $\Imrho \cap \tilde{\mc G}$ which act trivially on $C_B[W - W']$. 
If $|W'| = 1$, then by definition, $\pi_{W'}( \tilde{\mc G})$ is trivial. 
For $|W'| > 1$, we break the argument into 
two cases: $W'$ maximal and $W'$ nonmaximal. We let $\mcS(W') = \pi_{W'}(\tilde{\mc G})$.\\

	\noindent \textbf{Case 1:} $W'$ is maximal and $|W'| > 1$. \\
	This case requires a bit of detail since, as we will see, $\mcS(W') \cong \SL_m(\mcO)^{\text{op}} \ltimes \mcO^{m}$.
	We will use Proposition \ref{prop:transvgen} to get the $\SL$ factor and then additional unipotents to obtain the $\mcO^m$ factor.
	Since $U(W') -W' = \emptyset$, the representation $\pi_{W'}$ is just restriction to $C_B[W']$. Consequently, when discussing automorphisms
	of $C_B$ fixing $C_B[W - W']$, we may drop the notation $\pi_{W'}$.
	
	We choose a new basis for $C_B[W']$. By condition 1 of $\ddagger$, there exists $\mu \in W'$ such that $b_\mu$ is invertible. Let
	$e_\mu = \hat \mu_B$, and let $e_w = e_{w, \mu}$ for $w \in W' - \mu$. Then, 
	$\{e_w \mid w \in W'\}$ is a new basis for $C_B[W']$.
	Since it is only a change of finite index, we may assume that $\mcS(W')$ is the 
	group preserving the $\mcO$-span of this new basis.
	
	Let's describe an arbitrary $\psi \in \mcS(W')$. 
	Note that $e_u \in H_B$ for all $u \in W' - \mu$.
	Moreover, $\{e_u \mid u \in W'-\mu\}$ is a basis for $H_B[W']$.
	Thus, from the definitions, $\psi$ acts on $\mu, w \in W'$ as follows
	$$\psi(e_\mu) = e_\mu + \sum_{u \in W'-\mu} \beta_{u, \mu} e_u$$ 
	$$\psi(e_w) = \sum_{u \in W' - \mu} \beta_{u, w} e_u$$
	for some coefficients $\beta_{u, w} \in \mcO$. Moreover, since $H_B[W']$ is invariant and $\psi$ acts trivially
	on $C_B[W']/ H_B[W']$, the reduced norm of $\psi$ restricted to $H_B[W']$ must be $1$ by Lemma \ref{lemma:nrdonparabolic}.
	In view of the definition of $\tilde{ \mc G}$, any $\psi$ satisfying the above is in $\mcS(W')$.
	Let $\mathfrak H < \mcS(W')$ be the subgroup of automorphisms fixing $e_\mu$. Let 
	$\mathfrak N < \mcS(W')$ be the subgroup of automorphisms fixing $e_u$ for all $u \neq \mu$.
	The above establishes $\mcS(W') = \mathfrak N \cdot \mathfrak H $, and so it suffices to show
	$\Imrho$ contains finite index subgroups of $\mathfrak H$ and $\mathfrak N$.
	
	We start with $\mathfrak H$. By Claim \ref{claim:transvbyanycycle}, for all $v, u \in W'-\mu$ and
	$b \in M \mcO$ for $M \in \Z$
	sufficiently large, $\Imrho$ contains $T_{v, be_u}$. In our new basis, for $v, u \in W'-\mu$ one computes
	$$T_{v, be_u}(e_v) = e_v + be_u$$ $$T_{v, be_u}(e_w) = e_w \text{ for all } w \neq v.$$ 
	These are precisely the automorphisms required by Proposition \ref{prop:transvgen}
	where we view $\{e_u \mid u \in W'-\mu\}$ as the $e_i$. Thus, $\Imrho$ contains some finite index subgroup
	of $\mathfrak H$.
	
	Now, we consider $\mathfrak N$. By conditions $\ddagger$, there is some $v \in W'-\mu$ such that $b_v$ is invertible. Note that
	$e_{u, v} = e_u - b_u b_v^{-1} e_v$ for $u \in W'-\{\mu, v\}$. 
	Claim \ref{claim:transvbyanycycle} implies that $T_{\mu, b e_{u, v}} \in \Imrho$ 
	for all $u \in W' - \{\mu, v\}$ and for all $b \in M \mcO$ with $M$ sufficiently large.
	We compute
	$$\begin{array}{rcl} T_{\mu, be_{u, v}}(e_\mu) &=& e_\mu + b e_{u, v}\\
	  T_{\mu, be_{u, v}}(e_w) &=& e_w - b_w b_\mu^{-1} b e_{u, v} \text{ for all } w \in W'-\mu \end{array}$$
	By taking $M$ even larger, we can additionally ensure $T_{\mu, be_{u, v}} \in \mcS(W')$.
	For $z \in H[W']$,
	let $\phi_{z} \in \mcS(W')$ denote the automorphism defined by 
	$$  \phi_z(e_\mu) = e_\mu + z$$
	$$  \phi_z(e_w) = e_w \text{ if } w \neq \mu. $$
	Note that the restriction of $T_{\mu, be_{u, v}}$ to $H_B[W']$ must agree with the restriction
	of an element of $\mathfrak H$. Consequently, for $M$ sufficiently large
	(depending on the index of $\mathfrak H \cap \Imrho$ in $\mathfrak H$), we can also ensure
	there is some element of $\mathfrak H \cap \Imrho$ whose action on $H_B[W']$
	agrees with $T_{\mu, be_{u, v}}$. Precomposing by the inverse of that element, we find that 
	$\phi_{b e_{u, v}} \in \Imrho$ for $b \in M \mcO$ and $u \in W' - \{\mu, v\}$. 
	
	Now, for $u \in W' - \{\mu, v\}$, a computation yields 
	$[T_{u, b'e_v}, \phi_{b e_{u, v}}] = \phi_{bb' e_v}$, and 
	$[T_{v, b'e_w}, \phi_{b e_v}] = \phi_{bb' e_w}$ for all $w \in W' - \{\mu, v\}$.
	For some sufficiently large $M$ then,
	$\phi_{b e_w} \in \Imrho$ for all $b \in M \mcO$ and all $w \in W' - \{\mu\}$. I.e. $\Imrho$
	contains a finite index subgroup of $\mathfrak N$. \\
	
	\noindent \textbf{Case 2:} $W'$ is nonmaximal and $|W'| > 1$.\\ 
	By Lemma \ref{lemma:rednrmonsubquot}, $\pi_{W'}(\psi)$ has reduced norm $1$ for all $\psi \in \tilde{\mc G}$.
	Clearly, $\{\hat v_B \mid v \in W'\}$ projects to a basis for $C_B[U(W')]/C_B[U(W')-W']$, and its $\mcO$-span
	is preserved by $\pi_{W'}(\tilde{\mc G})$. These turn out to be the only restrictions
	on $\pi_{W'}(\tilde{\mc G})$ although it is not immediately clear from the definition of $\tilde{\mc G}$.
	In this case, we will prove that the group of such elements up to finite index is contained in
	$\pi_{W'}(\Imrho \cap \tilde{\mc G})$.
	 Let $e_v$ denote the projection of $\hat v_B$ to $C_B[U(W')]/C_B[U(W')-W']$.
	
	Let $\mu \in U(W')-W'$ be a maximal element such that $b_{\mu}$ is invertible. Since $\mcO$ is finitely generated
	and $W'$ is finite, there is a uniform integer $M > 0$ such that $T_{v, z} \in  \Imrho \cap \tilde{\mc G}$
	for all $z = b e_{u, \mu}$ for all distinct $v, u \in W'$ and $b \in M \mcO$. 
	For such $T_{v, z}$, the automorphism $\pi_{W'}(T_{v, z})$ 
	maps $e_w \mapsto e_w$ for $w \neq v$ and $e_v \mapsto e_v + b e_u$. These are precisely the elements required by
	Proposition \ref{prop:transvgen}, and thus $\pi_{W'}(\Imrho \cap \tilde{\mc G})$
	is commensurable to $\pi_{W'}(\tilde{\mc G})$. \\
	
	Note that in each of the above cases, we only appealed to elements of $\Imrho$ which acted trivially
	on $W - W'$. Consequently, we have proved that $\pi(\Imrho \cap \tilde{\mc G})$
	and $\pi(\tilde{\mc G})$ are commensurable. This establishes the theorem.

\subsection{Proof of the remaining main technical theorems}

\begin{proof}[Proof of Corollary \ref{cor:maintechcorlowerbound}]
	Apply Theorem \ref{thm:maintechthmlowerbound}.
\end{proof}

\begin{proof}[Proof of Corollary \ref{cor:reloutquots}]
	We want to use the representation $\rho_B$ to induce a homomorphism on some finite
	index subgroup of $\Out(F[W], \mc U)$. Since $\Aut(F[W], \mc U; R)$ has finite index image
	in $\Out(F[W], \mc U)$, it suffices to find some finite index subgroup $\Delta < \Aut(F[W], \mc U; R)$
	where $\Delta \cap \Inn(F[W])$ lies in the kernel of $\rho_B$. Then, 
	Corollaries \ref{cor:maintechcorupperbound} and \ref{cor:maintechcorlowerbound} finish the proof.
	
	Recall that $\rho_B \colon \Aut(F[W], \mc U; R) \to \Aut_B(H_B)$ is a restriction of 
	$\nu \colon  \Aut(F[W], \mc U; R) \to \Aut(\HH_1(R))$ and that $\nu(\varphi)$ is just the
	homomorphism induced by $\varphi|_R$. Let $I = \Inn(F[W]) \cap \Aut(F[W], \mc U; R)$.
	Since $R$ has finite index in $F[W]$, every inner automorphism of $F[W]$ has a power in
	$\Inn(R)$, but $\Inn(R)$ acts trivially on $\HH_1(R) = R/[R, R]$. I.e. $\rho_B(I)$
	is a torsion group. By Selberg's Lemma, $\Aut(\HH_1(R)) \cong \GL_t(\Z)$ contains a 
	finite index torsion-free subgroup. Then the preimage $\Delta$ of this subgroup in $\Aut(F[W], \mc U; R)$
	has the required property.
\end{proof}

\section{Some finite groups and their representations} \label{section:fingrps}
To finish the proof of Theorem \ref{thm:mainthm}, we need to find appropriate finite groups $G$ and representations $F[W] \to G$,
factoring through a free group quotient,
to plug in to our main technical corollaries (\ref{cor:maintechcorupperbound} and \ref{cor:maintechcorlowerbound}). 
Conditions 1 and 2 of $\ddagger$ are slightly complicated; we will instead
find a group $G$ and simple factor $B$ such that $b_v$ is invertible for all $v \in W$ and any pair $u, v \in W$
generates $G$. While that may seem a difficult condition to satisfy at first, we will show it is possible for the symmetric group,
alternating group
and finite Heisenberg group. Some indication that this is reasonable for the symmetric and alternating group is Dixon's 
Theorem \cite{Dixon} which states that, with probability approaching $1$ as $m \to \infty$,
two random elements of the symmetric group on $m$ letters generate either the symmetric group or its alternating
subgroup. Let $\Sym(m)$ denote the symmetric group on $m$ letters and $\Alt(m)$ its
alternating subgroup.

\begin{lemma} \label{lemma:gensym} Let $k$ be a positive integer.
	If $m \geq 6k+3$ and $m$ is even (resp. odd), then there are $m$-cycles $g_1, \dots, g_k \in \Sym(m)$ such that $g_i, g_j$ generate
	$\Sym(m)$ (resp. $\Alt(m)$) for all $i \neq j$.
\end{lemma}

\begin{remark} Our goal is simply to show that there are finite groups $G$ with the desired properties
	and not produce an exhaustive list. The above lemma is possibly far from optimal.
\end{remark}

\begin{proof}
	We will use cycle notation, but consider permutations as bijections on $\{1, \dots, m\}$ so that composition is composition
	of functions (from right to left). It is easy to check that $\Sym(m)$ (resp. $\Alt(m)$) is generated by 
	$\sigma = \begin{array}{ccccc}(1 & 2 & 3 & \dots & m)\end{array}$ and 
	$\begin{array}{ccc}(1 & 2 & 3)\end{array}$ when $m$ is even (resp. odd). We let 
	$\tau_i = \begin{array}{cc}(3i\!-\!2 & 3i\!-\!1)\end{array}$, let
	$g_1 = \sigma$ and let $g_i = \tau_i \sigma \tau_i$ for $1 < i \leq k$. We claim that $g_i, g_j$ generate
	for any $i \neq j$.
	
	First, consider $i = 1 \neq j$. Then,
	$$g_j g_1^{-1} = \tau_j \sigma \tau_j \sigma^{-1} = 
	 \begin{array}{ccc}(3j\!-\!2 & 3j\!-\!1 & 3j).\end{array}$$
	This $3$-cycle is conjugate to $\begin{array}{ccc}(1 & 2 & 3) \end{array}$ under a power of $g_1$.
	
	Now, suppose $i \neq 1 \neq j$ and $i < j$. Then,
	$$g_j g_i^{-1}  = \tau_j \sigma \tau_j \tau_i \sigma^{-1} \tau_i  = \tau_j (\sigma \tau_j \tau_i \sigma^{-1} \tau_i \tau_j) \tau_j.$$
	It suffices to show $\tau_j g_j \tau_j = \sigma$ and $\tau_j g_j g_i^{-1} \tau_j = \sigma \tau_j \tau_i \sigma^{-1} \tau_i \tau_j$ generate.
	We define and compute the following.
	$$\begin{array}{rcl} 
		\sigma_1 & =&  \sigma \tau_j \tau_i \sigma^{-1} \tau_i \tau_j = 
	 				\begin{array}{ccc}(3i\!-\!2 & 3i\!-\!1 & 3i) \end{array} \begin{array}{ccc}(3j\!-\!2 & 3j\!-\!1 & 3j)\end{array} \vspace{.2cm} \\
		 \sigma_2 & = & \sigma^{3j+2-3i} \sigma_1 \sigma^{-3j-2+3i} \vspace{.2cm} \\  
		 & =& \begin{array}{ccc}(3j & 3j\!+\!1 & 3j\!+\!2)\end{array} 
				\begin{array}{ccc}(6j\!-\!3i &  6j\!-\!3i\!+\!1 & 6j\!-\!3i\!+\!2)\end{array} \vspace{.2cm}\\ 
		\sigma_3 &= &(\sigma_1 \sigma_2)^6 = \begin{array}{ccccc}(3j\!-\!2 & 3j\!-\!1 & 3j & 3j+1 & 3j+2)\end{array} \vspace{.2cm}\\ 
		\sigma_4 &=& \sigma^{-3j+2} \sigma_3 \sigma^{3j-2} =
				\begin{array}{ccccc}(1 & 2 & 3 & 4 & 5)\end{array} \vspace{.2cm}\\ 
		\sigma_5 &=& \sigma^{-3} \sigma_4^4 \sigma \sigma_4^{-1} \sigma^{-1} \sigma_4^{-3} \sigma^3  
			= \begin{array}{ccc}(1 & 2& 3)\end{array} \end{array}$$
	Since $\sigma$ and $\begin{array}{ccc}(1 & 2& 3)\end{array} $ generate, we are done.
\end{proof}

\begin{lemma} \label{lemma:symstdrep}
		Let $q\colon \Sym(m) \to \GL_{m-1}(\Q)$ be the standard representation of the symmetric group.
		Then $q(\sigma) - I$ is invertible for all $m$-cycles $\sigma \in \Sym(m)$.
\end{lemma}
\begin{proof}
	Since all $m$-cycles are conjugate, it suffices to prove this for the $m$-cycle 
	$\begin{array}{ccccc}(1 & 2 & 3 & \dots& m)\end{array}$. Recall that the standard
	representation can be defined as follows. $\Sym(m)$ acts on $\Q^m$ in the obvious way by permuting the standard basis
	vectors $e_1, \dots, e_m$. The standard representation is the subrepresentation $V \subset \Q^m$ which is the span of
	$v_i = e_i - e_{i+1}$ for $1 \leq i < m$. Then, $q(\sigma)$ maps $v_i \mapsto v_{i+1}$ for $i < m-1$ and maps
	$v_{m-1} \mapsto -v_{m-1} - \dots - v_1$. A straightforward computation then shows that the determinant of
	$q(\sigma) - I$ is nonzero.
\end{proof}

To complete the picture for the symmetric and alternating groups, we recall the following result. We provide
a proof for convenience.

\begin{prop} \label{prop:Qsymfactor}
	Let $m \geq 7$. Then, the maps $\Q[\Sym(m)] \to \Mat_{m-1}(\Q)$ and $\Q[\Alt(m)] \to \Mat_{m-1}(\Q)$ linearly extending the 
	standard representation $\Sym(m) \to \GL_{m-1}(\Q)$ are surjective.
\end{prop}
\begin{proof}
	It suffices to prove this for $\Alt(m)$. By tensoring with $\C$, we obtain $\C[\Alt(m)] \to \Mat_{m-1}(\C)$
	which is the linear extension of the standard representation with coefficiencts in $\C$
	and is also irreducible over $\C$. Since 
	the simple factors of $\C[\Alt(m)]$ correspond to irreducible complex representations and the only
	finite-dimensional simple $\C$-algebras are $\Mat_k(\C)$ for some $k$, the representation
	$\C[\Alt(m)] \to \Mat_{m-1}(\C)$ is surjective. I.e. the image of $\Alt(m)$ spans $\Mat_{m-1}(\C)$, but
	then it must also span $\Mat_{m-1}(\Q)$.
\end{proof}

We now analyze the Heisenberg group, its generating sets and some of its irreducible representations. We denote by $H(k)$
the mod $k$ Heisenberg group: 
$$ H(k) = \{ C \in \Mat_3(\Z/k\Z) \;\; | \;\; C \text{ is upper unitriangular}\}$$
We let
$$ X = \left( \begin{array}{rrr} 1 & 1 & 0 \\ 0 & 1 & 0 \\ 0 & 0 & 1\end{array} \right) \;\;\;
 Y = \left( \begin{array}{rrr} 1 & 0 & 0 \\ 0 & 1 & 1 \\ 0 & 0 & 1\end{array} \right) \;\;\;
Z = \left( \begin{array}{rrr} 1 & 0 & 1 \\ 0 & 1 & 0 \\ 0 & 0 & 1\end{array} \right)$$

\begin{lemma} \label{lemma:Heisgens}
	 Let $p$ be the smallest prime factor of $k$. Then,
	the set $S = \{ X, Y X^s \;\; | \;\; 0 \leq s \leq p-1 \} $ is a set of maximal size where 
	all pairs of distinct elements $g_1, g_2 \in S$ generate $H(k)$.
\end{lemma}
\begin{proof}
	Note that the center $Z(H(k))$ is generated by $Z$.
	We first claim that $g_1 = Y^{s_1} X^{s_2} Z^{s_3}$ and $g_2 = Y^{t_1} X^{t_2} Z^{t_3}$ generate $H(k)$ if and only 
	if they generate $H(k)/Z(H(k)) \cong (\Z/k\Z)^2$. Clearly, it's necessary. If $g_1, g_2$ generate 
	$H(k)/Z(H(k))$, then since $X, Y$ are generators of $H(k)/Z(H(k))$, it must be the case that 
	$s_1 t_2 - t_1 s_2$ is a unit in $\Z/k\Z$. By computation, $[g_1, g_2] = Z^{s_1 t_2 - t_1 s_2}$,
	and so $Z(H(k))$ is in the group generated by $g_1, g_2$.
	
	If $p = p_1, \dots, p_t$ are the prime factors of $k$, then $g_1, g_2$ generate $H(k)/Z(H(k)) \cong (\Z/k\Z)^2$
	if and only if they generate after passing to the quotient $(\Z/k\Z)^2 \to (\Z/p_i\Z)^2 $ for all $i$.
	The maximal size of a set of vectors, each pair of which generates $(\Z/p_i\Z)^2$ is $p_i + 1$.
	Since $p$ is the smallest prime, the largest subset of $H(k)$, each pair of which generates $H(k)$,
	has size at most $p+1$. The set $S$ has size $p+1$ and each pair in it generates $H(k)$ by the above discussion.
\end{proof}

\begin{lemma} \label{lemma:Heisreps}
	Let $S= \{ X, Y X^s \;\; | \;\; 0 \leq s \leq p-2 \} \subset H(k)$. Let $1 < m < k$ be a factor of $k$ which is coprime
	to $\ell = \frac{k}{m}$, and let $\zeta$ be a primitive $k$th root of unity. 
	There exists a factor $B \cong \Mat_{m}(\Q(\zeta))$ of the group
	ring $\Q[H(k)]$ such that for every $g \in S$ the element $b - 1_B \in B$ is invertible where $b$ is the
	projection of $g$ to $B$.
\end{lemma}
\begin{proof}
	We first use the following representation $r\colon H(k) \to \GL_{m}(\Q(\zeta))$ defined on generators by:
	$$ r(X) = \left(\begin{array}{rrrr} \zeta & & & \\ & \zeta^{\ell+1} & & \\ & & \ddots & \\ & & & \zeta^{(m-1)\ell+1}
				\end{array}\right) $$
	$$	r(Y) = 	\left(\begin{array}{rrrr}  & & & \zeta^m \\ 1 &  & & \\ & \ddots & & \\ & & 1 & 
							\end{array}\right) \hspace{1.7cm}
		r(Z) = \zeta^{\ell} I$$
	This gives a well-defined representation since $r(X), r(Y), r(Z)$ satisfy the following relations in a presentation
	of $H(k)$ : $X^k = I, Y^k = I, Z^k = I, XZ = ZX, YZ = ZY, XYX^{-1}Y^{-1} = Z$.
	We will show that the induced homomorphism $r\colon  \Q[H(k)] \to \Mat_{m}(\Q(\zeta))$ is surjective and hence the latter
	ring is a factor of $\Q[H(k)]$.
	
	Let $H' < H(k)$ be the subgroup generated by $X, Z$; this subgroup is isomorphic to $(\Z/k\Z)^2$.
	Consequently, $\Q[H']$ is isomorphic to a product of fields.
	Since $r$ maps $X, Z$ to diagonal matrices, the resriction of $r$
	to the subring $\Q[H']$ decomposes as a product of representations
	$r_i\colon  \Q[H'] \to \Q(\zeta)$ where $r_i(X) = \zeta^{\ell i + 1}$ and $r_i(Z) = \zeta^{\ell}$ for $0 \leq i \leq m-1$. 
	Each $r_i$ is surjective, so each $r_i$ is equivalent to projection onto a factor of $\Q[H']$. Moreover,
	each $r_i$ is projection to a different factor since the kernels are all distinct. Indeed for $i \neq j$,
	$$ r_i(X^{\ell} - Z^{\ell i+1}) = 0 \neq r_j(X^{-\ell} - Z^{\ell(i-1)+1}) $$
	Consequently, the product of representations $\prod_i r_i$ is surjective, or equivalently, $r(\Q[H'])$
	is the set of all diagonal matrices in $\Mat_{m}(\Q(\zeta))$. Every other matrix is a sum of the form
	$\sum_{i =0}^{m -1} d_i r(Y)^i$ where $d_i$ is a diagonal matrix. Thus, $r$ is surjective and
	$\Mat_{m}(\Q(\zeta))$ is a factor of $\Q[H(k)]$.
	
	Now suppose $g \in S$ and $b = r(g)$. The matrix $b - I$
	is invertible if and only if $b$ does not have $1$ as an eigenvalue. This is clear if $g = X$, so we assume
	$g = Y X^s$ for some $0 \leq s \leq p-2$. In this case,
	$$b = 	\left(\begin{array}{rrrrr}  & &  & & \zeta^{m + s(m-1)\ell  + s} \\ \zeta^s & & & & \\
	 		& \zeta^{s\ell+s} & & & \\ & & \ddots & & \\ & & & \zeta^{s(m-2)\ell + s} 
						\end{array}\right)$$
	The characteristic polynomial of $b$ is $\lambda^{m} \pm \zeta^t$ where
	$t = sm + s \ell {m\choose2} + m$. 
	Since $0 \leq s \leq p-2$, we have $s+1$ is coprime to $\ell$. Thus, $t \equiv (s+1)m \not\equiv 0$ (mod $\ell$), 
	and in particular $\zeta^t \neq 1$. Thus, $1$ cannot be an eigenvalue of $b$.
\end{proof}

\section{Proof of the Main Theorems} \label{section:proofofmainthms}

We are now ready to prove our main theorems.

\begin{proof}[Proof of Theorem \ref{thm:mainthm}]
	In view of Corollary \ref{cor:generaloutagamquots}, we need to do two things. First, we need to find
	$G, B, q$ that together with $\preceq$ satisfy $\ddagger$. Second, we need to show that the group $\mc G$
	as defined in Corollary \ref{cor:generaloutagamquots} is virtually isomorphic to the group $\mc G$
	as defined in Theorem \ref{thm:mainthm}.
	From condition 3 of $\dagger$, we see that the first part of condition 3 of $\ddagger$ is satisfied.
	
	First, consider the case when $m \geq 6|W| + 2$ and $k = 1$. I.e. $\mcO = \Mat_m(\Z)$.
	Let $G = \Sym(m+1)$ (resp. $G = \Alt(m+1)$) if $m+1$ is even (resp. odd). 
	We define $q\colon  F[W] \to G$ by mapping the basis $W$ to 
	the elements $g_1, \dots, g_{|W|}$ provided by Lemma \ref{lemma:gensym}.
	Lemmas \ref{lemma:gensym} and \ref{lemma:symstdrep}  and Proposition \ref{prop:Qsymfactor} 
	imply that $G, B, q, \preceq$ satisfy $\ddagger$ with the homomorphism $p:\Q[G] \to \Mat_{m}(\Q) = B$ 
	induced by the standard representation $G \to \GL_{m}(\Q)$. Indeed, 1 of $\ddagger$
	holds since $p(q(u) - 1) = 1_B(q(u)-1)$ is invertible for all $u \in W$ and 2 holds since
	every pair of elements in $W$ generates $G$.
	
	Now, consider the case when $\mcO = \Mat_m(\Z(\zeta))$ where $m$ divides $k$, the largest
	prime factor of $k$ is at least $|W|$, and $m, \frac{k}{m}$ are coprime. 
	In this case, let $G = H(k)$, the mod $k$ Heisenberg group. We define
	$q\colon  F[W] \to H(k)$ by mapping $W$ injectively into the set $S$ from Lemma \ref{lemma:Heisreps}.
	We let $B = \Mat_{m}(\Z[\zeta_k])$ be the factor of $\Q[G]$ from
	Lemma \ref{lemma:Heisreps}. 
	Then, $(G, B, q)$ satisfy conditions 1 and 2 of $\ddagger$ by Lemmas \ref{lemma:Heisgens} and \ref{lemma:Heisreps}.
	Since $m \geq 2$, condition 3 is satisfied as well. 
	
	We now show that in both cases, $\mc G = \mc G_1$ as defined for Corollary \ref{cor:generaloutagamquots}
	and $\mc G = \mc G_2$ as defined for Theorem \ref{thm:mainthm} are virtually isomorphic.
	In all cases of $G, B, q$ above, $b_v = 1_B(q(v) - 1)$ is invertible for all $v \in W$. We can
	therefore define a new basis of $C_B$ as $e_v = b_v^{-1} \hat v_B$. 
	Though $\mc G_1$ is the subgroup
	of $\mc G(B)$ preserving the $\mc O$-span of the $\hat v_B$, it is commensurable to the
	subgroup $\mc G_3$ of $\mc G(B)$ preserving the $\mc O$-span of the $e_v$.
	Moreover, the set of vectors $\{e_v - e_w \mid v, w \in W\}$ spans $H_B$.
	It is now clear that $\mc G_3 \cong \mc G_2$.
\end{proof}

\begin{proof}[Proof of Theorem \ref{thm:mainthmsemidprod}]
	Let $G, B, q, \preceq,$ and $\mc U$ as in the proof of Theorem \ref{thm:mainthm}.
	We begin by identifying a subgroup of $\Aut_B(C_B)$ which is isomorphic to $\mc G \ltimes \mcO^{n-1}$.
	Let $e_v = b_v^{-1} \hat v_B$ as in the previous proof, let $C_\mcO$ be the 
	$\mcO$-span of the $e_v$, and let $H_\mcO$ be the $\mcO$-span of $\{e_v - e_w \mid v, w \in W\}$.
	Let $\tilde{\mc G}(B)$ be the subgroup of $\Aut_B(C_B)$ as defined in Section \ref{section:definingrho}.
	Then, the subgroup of $\tilde{\mc G}(B)$ preserving $C_\mcO$ and its projection to $\Aut_\mcO(H_\mcO)$
	are isomorphic to $\tilde{ \mc G}$ and $\mc G$ respectively as defined in the introduction and
	Theorem \ref{thm:mainthmsemidprod}. We let $\mc N < \Aut_B(C_B)$ be the subgroup preserving
	$H_B$ and $C_\mcO$ and acting trivially on $H_B$ and $C_B/H_B$. Let $\mathfrak G = \tilde{\mc G} \mc N$.
	Then $\mathfrak G$ projects to $\mc G$ in $\Aut_B(H_B)$ with kernel $\mc N$ and 
	$\mathfrak G \cong \mc G \ltimes H_\mcO \cong \mc G \ltimes \mcO^{n-1}$.
	
	Now, let $\tilde \rho_B: \Aut(F[W]; R) \to \Aut_B(C_B)$ be the representation
	as defined in Section \ref{section:isotypcomps}. By Lemma \ref{lemma:AutAGamtoAutFn},
	it suffices to prove that $\tilde \rho_B(\Delta)$
	is commensurable with $\mathfrak G$ where $\Delta$ is the intersection of
	$\Aut(F[W], \mc U; R) \Inn(F[W])$ and $\Aut(F[W]; R)$. As for $\mc G$ in the previous
	proof, $\tilde{\mc G}$ is commensurable to the group of
	the same name in Theorems \ref{thm:maintechthmupperbound} and \ref{thm:maintechthmlowerbound}, and so 
	by those same theorems, $\tilde \rho_B(\Aut(F[W], \mc U; R))$ is commensurable
	to $\tilde{\mc G}$ which projects to a finite index subgroup of $\mc G$.
	
	It now suffices to prove that $\tilde \rho_B(\Inn(F[W]) \cap \Delta)$ is commensurable to $\mc N$.
	Since $e_v \equiv e_w$ (mod $H_B$), $\mc N$ consists entirely of automorphisms $\psi$ satisfying
	$\psi(e_v) - e_v = \psi(e_w) - e_w \in H_\mcO$ for all $v, w \in W$.
	Clearly, $\Inn(R)$ is of finite index in $\Delta$. Since $\Inn(R)$ acts trivially on $\HH_1(R; \Q) = H$ and thus $H_B$,
	the image $\tilde \rho_B(\Inn(R))$ virtually lies in $\mc N$. If $\varphi_r$ is conjugation by $r \in R$,
	then 
	$$\tilde \rho_B(\varphi_r)(e_v) = b_v^{-1} \cdot 1_B \cdot \overline{rvr^{-1}} = e_v - 1_B \overline{r}$$
	Since $H_\mcO$ is finitely generated, by Lemma \ref{lemma:lifttogroupring}, 
	there is some large integer $M > 0$ such that all $z \in M H_\mcO \subset H$
	are equivalent to $1_B \overline{r}$ for some $r \in R$. Since, as abelian groups, $\mc N \cong H_\mcO$,
	we conclude that $\tilde \rho_B(\Inn(R))$ contains a finite index subgroup of $\mc N$.
\end{proof}

\begin{remark} \label{remark:mfGistildeG}
	Consider the case where $v \preceq w$ for all $v, w \in W$ and $|W| \geq 4$. Then $\mc U = \{ \emptyset, W\}$ and $\Aut(F[W], \mc U) = \Aut(F[W])$.
	Moreover, conditions 1, 2, and 3 defining $\tilde{\mc G}$ are vacuous, and so $\mc N$ is a subgroup of $\tilde{\mc G}$.
	Consequently, the proof above implies that $\tilde \rho_B(\Aut(F[W]; R))$ is commensurable with $\tilde{\mc G}$. 
	Moreover, if $\mc O = \Mat_m(\Z)$, then $\tilde{\mc G}$ is virtually isomorphic to $\mc G \ltimes H_\mcO$ which is
	virtually isomorphic to 
	$$\SL_{n-1}(\Mat_m(\Z)^{op}) \ltimes (\Mat_m(\Z))^{n-1} \cong \SL_{n-1}(\Mat_m(\Z)) \ltimes (\Mat_m(\Z))^{n-1} $$
	$$ \cong \SL_{(n-1)m}(\Z) \ltimes \Mat_{(n-1)m \times m}(\Z) $$
	 However, the action of $\SL_{(n-1)m}(\Z)$ preserves the columns of matrices in $\Mat_{(n-1)m\times m}(\Z)$, and so
	 there is a surjective map $$\SL_{(n-1)m}(\Z) \ltimes \Mat_{m(n-1)\times m}(\Z) \to \SL_{(n-1)m}(\Z) \ltimes \Z^{(n-1)m}.$$
\end{remark}

\begin{proof}[Proof of Theorem \ref{thm:outfnmodtrans}]
	If we choose the preorder $\preceq$ where $u \preceq v$ for all $u, v \in W$, then $\Aut(F[W], \mc U)$,
	and consequently $\Aut(F[W], \mc U; R)$, is a finite index subgroup of $\Aut(F[W])$.
	We can therefore use Corollary \ref{cor:maintechcorlowerbound} and results from Section \ref{section:upperbound}.
		
	Before that, we use results of \cite{MPut} to find an appropriate $G, B, q$ satisfying $\ddagger$. 
	Recall that each
	irreducible $\Q$-representation of a finite group $G$ corresponds to a simple factor $B$ of $\Q[G]$.
	Proposition 4.3 of \cite{MPut} states the following in the notation of this paper. Suppose $q: F[W] \to G$ is surjective,
	and $G$ has an irreducible representation $p: G \to \GL_t(\Q)$ where $p(q(x))$ has no fixed vector for all primitives
	$x \in F_n$. Let $p: \Q[G] \to \Mat_t(\Q)$ be the induced map and $B = p(\Q[G]))$ which, as mentioned, must be a simple
	factor of $\Q[G]$. Let $\rho_B$ be the representation corresponding to $G, B, q$, and let $k$ be the least common multiple
	of the orders of $q(x)$ over all primitives $x$. Then, for any transvection $T \in \Aut(F[W])$, the power $T^k$
	lies in the kernel of $\rho$. The proof of Theorem D in \cite{MPut} shows that there is such a homomorphism
	$q: F[W] \to G$ and irreducible representation $p$ when $n \geq 2$. We have produced our virtual representation
	$\rho_B$ whose kernel contains all $k$th powers of transvections. It remains to show that the image is an arithmetic group, namely 
	$\mc G$ as defined in Section \ref{section:definingrho}.
	
	To show that $\rho_B(\Aut(F[W], \mc U; R)$ virtually contains $\mc G$, we check that
	conditions $\ddagger$ are satisfied and use Corollary \ref{cor:maintechcorlowerbound}. 
	Conditions 2 and 3 follow easily since there is precisely one maximal
	equivalence class of size $n \geq 4$. The fact that $p(q(v))$ has no nontrivial fixed vectors
	implies that $1$ is not an eigenvalue and hence $p(q(v))-1_B = 1_B(q(v) - 1)$ is invertible. This holds for all $v \in W$
	and thus condition 1 holds.
	
	Recall from Section \ref{section:upperbound} that $\tilde{\mc G}'$ is the subgroup of $\Aut_B(C_B)$
	satisfying all the constraints of $\tilde{\mc G}$ except the condition on reduced norms.
	Let $\mc G'$ be the projection to $\Aut_B(H_B)$.
	By Corollary \ref{cor:upperboundworednrm}, $\rho_B(\Aut(F[W], \mc U; R))$ lies in $\mc G'$.
	Since $\mc U = \{\emptyset, W\}$, we have $\mc G' \cong \GL_{n-1}(\mcO^{op})$, and the only
	difference between $\mc G$ and $\mc G'$ is that elements of $\mc G$ must have reduced norm $1$.
	
	To close the gap between $\mc G'$ and $\mc G$, we may do one of the following. We can use 
	Corollary \ref{cor:maintechcorupperbound} and the fact that $\Aut(F_m)$ has property (T) for $m \geq 5$ \cite{KKN, KNO}. Alternatively,
	we can borrow a trick from Section 7 of \cite{GL}. 
	Let $K$ be the center of the algebra $B$. Then, $K$ is a finite field extension of $\Q$, and the image
	of the reduced norm map $\nrd: \mc G' \to K^\times$ lies in the (group of units of the) ring of integers of $K$.
	By Dirichlet's Unit Theorem then, the image $\nrd(\mc G')$ is a finitely generated abelian group. Moreover,
	letting $S < \mc G'$ be scalar matrices with entries in $K$, we have $\nrd(S) < \nrd(\mc G')$ is of finite index.
	Let $\Delta < \nrd(S)$ be a finite index free abelian subgroup, and
	let $\Delta' = \nrd^{-1}(\Delta) < \mc G'$.
	Picking some section $\sigma: \Delta \to S$, we then define a homomorphism $\Delta' \to \mc G$
	via $M \mapsto M s(\nrd(M)^{-1})$. Note that $\sigma$ is the identity on $\Delta' \cap \mc G$. Composing by this
	homomorphism, we get the desired virtual surjective homomorphism $\Aut(F[W]) \to \mc G$.
	Since $n \geq 4$, the group $\mc G$ contains a finite index subgroup of $\SL_2(\Z)$ which contains
	nonabelian free groups.
\end{proof}

\subsection{Other finite groups and arithmetic quotients}

We expect that characterizing those triples $(G, B, q)$ which satisfy conditions $\ddagger$ in any meaningful way is
difficult if not impossible. Even characterizing those algebras $B$ which appear as a factor of $\Q[G]$
cannot be described in a simple way. (Such $B$ up to Brauer-equivalence are, in some sense, characterized.
See section 9.8 of \cite{GLLM} for a brief discussion.) Consequently, we only have gone so far as to
show that the virtual arithmetic quotients obtainable via our theorems are quite varied
and may have arbitrarily large dimension.

\section{Domination Poset and the Conditions $\dagger$} \label{section:conditions}
In this section, we discuss how restrictive the conditions $\dagger$ are.
We briefly discuss condition 3. If we do not assume this condition, then the proofs of
our main theorems and main technical theorems simply do not hold. Specifically, we cannot
get the virtual representations from $\Aut(A_\Gamma)$ by considering the action on the homology
of a finite index subgroup. One runs into exactly the same issues as noted in \cite{GL}[Section 9.1].
On the other hand, in such a situation, $\Aut(A_\Gamma)$ maps onto $\Aut(F_2)$ which is virtually a free
group, and so $\Aut(A_\Gamma)$ virtually maps onto any finitely generated group \cite{GuirardelSale}.

We now discuss the remaining conditions. Condition 1 seems rather restrictive. However, since the domination
relation is so strong, as soon as two vertices are adjacent, it's necessarily the case that many other
vertices will become adjacent if the domination poset has some complexity. What will be shown is that
any RAAG with a somewhat complex domination poset either has a subset of vertices satisfying $\dagger$
or has very large cliques. While this doesn't rule out that $\Aut(A_\Gamma)$ has many virtual arithmetic
quotients, it does make it difficult to see how they come about from transvections.

\begin{lemma} \label{lemma:forcedadjacency}
	Let $\Gamma$ be a graph, and suppose $v_1, v_2 \in V(\Gamma)$ are distinct vertices.
	If $v_1$ and $v_2$ are adjacent, then every vertex of $U(v_1)$ is adjacent to every vertex of $U(v_2)$.
	In particular, if $v_1 \geq v_2$, then additionally $U(v_1)$ spans a complete graph.
\end{lemma}
\begin{proof}
	Clearly, all vertices in $U(v_1)$ are adjacent to $v_2$. Since every vertex in $U(v_2)$ dominates $v_2$,
	we conclude every vertex in $U(v_2)$ is adjacent to every vertex in $U(v_1)$. The last statment follows
	since $U(v_1) \subseteq U(v_2)$.
\end{proof}

We define a {\it domination chain} of a graph $\Gamma$ to be a subset $v_1, v_2, \dots, v_k \in V(\Gamma)$
satisfying $v_1 \geq v_2 \geq \dots \geq v_k$. Note that we allow there to be additional relations between
the $v_i$. E.g. $v_1$ and $v_2$ could be domination equivalent.

\begin{lemma} \label{lemma:conditionssatforLBofmax}
	Let $v \in V(\Gamma)$. If $W = L(v)$
	is an independent set, then $W$ satisfies conditions 1, 4, and 5 of $\dagger$. 
\end{lemma}
\begin{proof}
	Conditions 1 and 5 are obviously satisfied. Now suppose $v_1 \in W$ and $C \subset \Gamma - \st(v_1)$
	is a component of $\Gamma - \st(v_1)$. Suppose $C$ contains $v_2 \in W - L(v_1)$. Since $v$ dominates $v_2$, but
	$v_1$ does not, there is some $v_3$ which is adjacent to both $v$ and $v_2$ but not $v_1$. Consequently, $v, v_2, v_3 \in C$
	and, since $v_2$ was arbitrary, $C \supseteq W - L(v_1)$. Suppose instead that $C$ contains $v_2 \in L(v_1) - v_1$.
	Then, since $v_1$ dominates $v_2$, we have $C = \{v_2\}$. Consequently, $v_1$ divides $W$ trivially.
\end{proof}

\begin{prop} \label{prop:exceptional}
	Suppose $\Gamma$ is a graph where no subset $W \subset V(\Gamma)$ satisfies conditions 1, 2, 4, and 5 of $\dagger$.
	Then, for every domination chain $v_1 \geq v_2 \geq v_3$ of length $3$, $U(v_1)$ is a clique in $\Gamma$,
	and $v_1$ dominates an adjacent vertex.
\end{prop}
\begin{proof}
	To prove the first part it suffices to show that if 
	$v_0 \geq v_1 \geq v_2 \geq v_3$ is a domination chain, then $v_0$ is adjacent
	to $v_1$. Suppose they aren't adjacent. Then, $v_0$ is not adjacent to $v_2$ or $v_3$ since $v_1$ dominates them, 
	and $v_1, v_2, v_3$ is an independent set or else $v_0$ and $v_1$ are adjacent by Lemma \ref{lemma:forcedadjacency}.
	Thus, $v_0, v_1, v_2, v_3$ is an independent set. We now claim $L(v_1)$ is an independent
	set. Given any distinct $u_1, u_2 \in L(v_1)$, we have $v_0, v_1 \in U(u_1) \cap U(u_2)$, but
	$v_0$ and $v_1$ are not adjacent, and so by Lemma \ref{lemma:forcedadjacency}, $u_1$ is not adjacent to $u_2$.
	By Lemma \ref{lemma:conditionssatforLBofmax}, $L(v_1)$ satisfies conditions 1, 2, 4, and 5, contradicting the assumptions.
	
	Now, suppose that there were a domination chain $v_1 \geq v_2 \geq v_3$ spanning an independent set in $\Gamma$.
	If $L(v_1)$ were an independent set, then $L(v_1)$ would satisfy conditions 1, 2, 4, and 5
	by Lemma \ref{lemma:conditionssatforLBofmax}. Consequently, there are two adjacent vertices in $L(v_1)$,
	and by Lemma \ref{lemma:forcedadjacency}, $v_1$ dominates an adjacent vertex. 
\end{proof}

\subsection{Some graphs containing vertex subsets satisfying $\dagger$}
One might wonder whether there are many graphs $\Gamma$ containing vertex sets satisfying $\dagger$.
In fact, we can show such graphs are as diverse as preorders satisfying 2 and 3 of $\dagger$.

\begin{prop} \label{prop:arbpreorder}
Given any pair $(X, \preceq)$ of a nonempty set $X$ and preorder, there is a graph
$\Gamma$ and subset $W \subseteq V(\Gamma)$ such that $(W, \preceq) \cong (X, \preceq)$ where $W$ is equipped with
the restriction of domination relation. Moreover, $W$ satisfies conditions 1, 4, and 5 of $\dagger$, and every vertex
of $V(\Gamma) - W$ is incomparable to all other vertices.
\end{prop}
Note that since conditions 2 and 3 depend only on the relation, $(W, \preceq)$ satisfies them if and only
if $(X, \preceq)$ does.

\begin{proof}
	The vertex set $V(\Gamma)$ will consist of four sets of vertices $W, V, V', U$. 
	The vertices of $W$ are copies of 
	the elements of $X$, i.e. for every $x \in X$ there is exactly one $w_x \in W$.
	Similarly, the sets $V$ and $V'$ consists of copies $v_x, v_x'$ respectively of 
	each element of $X$. The set $U$ has three vertices $u_1, u_2, u_3$. 
	The edges of $E(\Gamma)$ are as follows. Each $w_x \in W$ is adjacent to $v_y \in V$ if and only if
	$x \succeq y$. Each vertex $v_y$ is adjacent to its copy $v_y'$. All vertices of $W$ are connected to the vertex $u_1$, and all vertices of $V'$
	are connected to $u_2$, and the vertex $u_3$ is only adjacent to $u_1$ and $u_2$. Any pair of vertices not mentioned are not adjacent.
	By construction, it is clear that $W$ satisfies condition 1 of $\dagger$.

	It is straightforward to check that no vertex in $W, V, V', U$ dominates a vertex in a different set. (E.g. no
	vertex from $W$ dominates a vertex in $V$.) It is also easy to see that vertices of $V'$ (resp. $U$) do not dominate
	vertices of $V'$ (resp. $U$). For all pairs, $v_x, v_y \in V$, the vertex $v_y$ is adjacent to $v_y'$ but $v_x$ is not,
	so $v_x$ does not dominate $v_y$. The neighbors of $w_x$ are $\{v_y \mid y \preceq x\} \cup \{u_1\}$, and so 
	$w_y$ dominates $w_x$ if and only if $y \succeq x$. Consequently, $W$ with the domination relation
	is isomorphic to $(X, \preceq)$ and satisfies condition 5 of $\dagger$.
	
	Now consider $\Gamma - \st(w_x)$ for $w_x \in W$. Suppose $w_y \in W - L(w_x)$. 
	Then, by construction, $x$ is not an upper bound of $y$, and so $w_x$ is not adjacent to $v_y$.
	Consequently, there is a path from $w_y$ to $u_2$ in $\Gamma - \st(w_x)$, and it follows that
	$W - L(w_x)$ lies in a single component of $\Gamma - \st(w_x)$. I.e. condition 4 is satisfied.
\end{proof}

\section{Appendix} \label{section:appendix}
In this appendix, we describe a couple examples of $\mc G$ and compare these groups to the image of $\rho_0$.
From Proposition \ref{prop:arbpreorder}, given any finite set $(W, \preceq)$ with a preorder, it can be realized as 
a subset of some $V(\Gamma)$ where the preorder agrees with the domination relation on $W$ and all other
vertices are incomparable. We therefore just present the preorder on $W$ (as a poset on equivalence classes)
and the graph $\Gamma$ will be the one produced by the proposition.

\begin{figure}[h]
	\centering
	\labellist
	\pinlabel $v_1$ at 14 141
	\pinlabel $v_2$ at 14 95
	\pinlabel $v_3$ at 14 49
	\pinlabel $v_4$ at 14 3
	\pinlabel $W_1$ at 96 107
	\pinlabel $W_2$ at 163 107
	\pinlabel $W_3$ at 230 107
	\pinlabel $W_4$ at 96 37
	\pinlabel $W_5$ at 163 37
	\pinlabel $W_6$ at 230 37
	\endlabellist
	\includegraphics{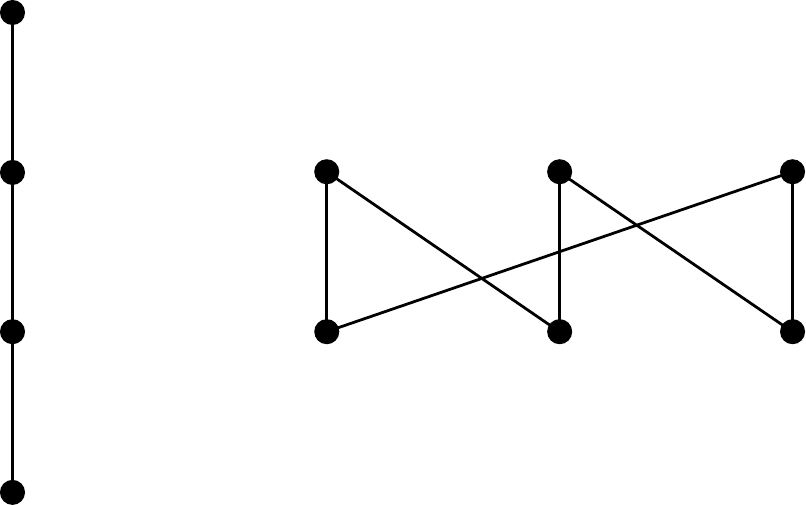}
	\caption{In the first example, the $v_i$ are vertices. In the second example, the $W_i$ are equivalence classes of size $3$. A vertex
	or equivalence class connected to a vertex or class below it indicates it is an upper bound of the vertex/class below it. }
	\label{figure:posets}
\end{figure}
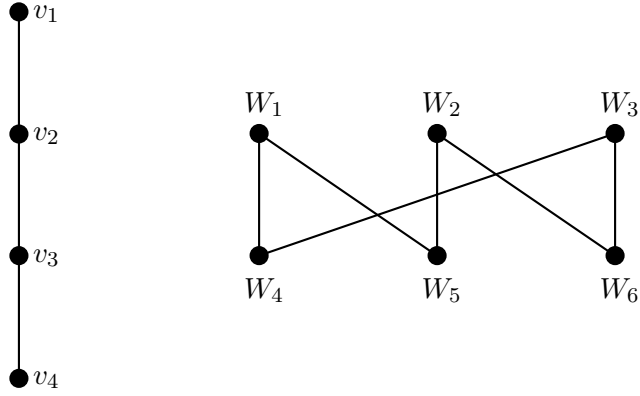

Let $\Gamma$ be the graph corresponding to the first preorder and $\Gamma'$ to the second preorder in Figure \ref{figure:posets}. 
We first describe the image
under the representation $\rho_0$. If we order our basis of $\HH_1(A_\Gamma)$ as $v_1, v_2, v_3, v_4$ followed by the rest,
then $\rho_0(\Aut(A_\Gamma))$, up to finite index, consists only of matrices of the following form.
$$
\begin{array}{c|c}
	\left(\begin{array}{c|c} 
				\begin{array}{cccc} \pm 1 & c_{12} & c_{13} & c_{14} \\
								 &  \pm 1 & c_{23} & c_{24} \\
								 &  & \pm 1 & c_{34} \\
								 &  &  & \pm 1
				\end{array} &  \\ \hline \\
			&  	\begin{array}{cccc}  \pm 1 &  &  &  \\
										       &  \pm 1 &  &  \\
										    & & \ddots & \\
										     &  & &  \pm 1
				\end{array}
			\end{array}
	\right) \;\;\; 
	& \;\;\; c_{ij} \in \Z 
\end{array}
$$
If we order our basis of $\HH_1(A_{\Gamma'})$ as $W_1, W_2, W_3, W_4$ followed by the rest,
then $\rho_0(\Aut(A_{\Gamma'}))$, up to finite index, consists only of matrices of the following form.
$$
\begin{array}{c|c}
	\left(\begin{array}{c|c} 
				\begin{array}{cccccc} D_{1} &  &  & C_{14} & C_{15} &  \\
									 & D_{2} &  &  & C_{25} & C_{26} \\
									 &  & D_{3} & C_{34} &  & C_{36} \\
									 &  &  & D_4 &  &  \\
									 &  &  &  & D_5 &  \\
									 &  &  &  &  &  D_6 
				\end{array} &  \\ \hline \\
				 &  \begin{array}{ccc} \pm 1 &  &    \\
											       & \ddots &   \\
											     &  &  \pm 1
					\end{array}
			\end{array}
	\right) \;\;\; 
	&
	\begin{array}{c} D_i \in \GL_3(\Z) \\ C_{ij} \in \Mat_3(\Z)
	\end{array}
\end{array} 
$$

Now consider $\mc G$ for $\mcO = \Mat_m(\Z)$. Recall that $\mc G$ is a subgroup 
of $\Aut_{\mcO}(H_{\mcO}) \cong \Aut_{\mcO}(\mcO^{n-1}) \cong \GL_{n-1}(\mcO^{op})$. 
Since $\mcO^{op} \cong \mcO$ via the transpose map, we can view $\mc G$ as a subgroup of 
$\GL_{n-1}(\mcO) \cong \GL_{(n-1)m}(\Z)$. Let $e_v$ denote the free generators 
of $C_\mcO$, and recall that $H_\mcO = H_\mcO^-$, and so the vectors $e_v - e_w$
generate $H_\mcO$. For the first example of graphs, we choose
as basis of $H_{\mcO}$, the vectors $e_i' = e_{v_i} - e_{v_1}$ for $2 \leq i \leq 4$.
Recall that $\mc G$ is the restriction of $\tilde{\mc G} < \Aut_{\mcO}(C_{\mcO})$ to $H_{\mcO}$. 
For any $\varphi \in \tilde{\mc G}$, we have $\displaystyle \varphi(e_{v_i}) = e_{v_i} + \sum_{j = 1}^{i-1} \beta_j e_{v_j}$,
for some $\beta_j \in \mcO$,
and so $\mc G$ consists of matrices of the following form.

$$
\begin{array}{c|c}
 	\left( \begin{array}{ccc}  I & C_{12} & C_{13} \\
								  & I & C_{23} \\
								  &  & I
			\end{array} 
	\right) \;\;\; 
	& \;\;\; C_{ij} \in \Mat_m(\Z)
\end{array}
$$

Now we choose a basis for $H_{\mcO}$ for the second example. There does not appear to be any canonical
choice of basis which results in a simple description of $\mc G$.
We first choose one vertex $\mu_i \in W_i$ from each of the maximal equivalence classes.
We define sets of vectors as follows. For $1 \leq i \leq 3$, let $E_i = \{e_v - e_{\mu_i} \mid v \in W_i - \mu_i\}$.
Let $F$ be the ordered set $\{e_{\mu_2}-e_{\mu_1}, e_{\mu_3}-e_{\mu_1}\}$. For $4 \leq i \leq 6$, let $E_i = \{e_v - e_{\mu_{i-3}} \mid v \in W_i\}$.
We then order the basis elements as follows: $E_1, E_2, E_3, F, E_4, E_5, E_6$. (We didn't order the $E_i$, but it doesn't affect
the form of the matrix.) With this ordered basis, $\mc G$ consists of matrices of the following form.

$$
	\left (\begin{array}{ccc|cc|ccc} 
					D_1  &       &        & C_{14} & C_{14} & C_{16} & C_{17} &  \\
					     & D_2   &        & C_{24} &        &        & C_{27} & C_{28} \\
					     &       & D_3    &        & C_{35} & C_{36} &        & C_{38} \\ \hline
					     &       &        &  I     &        &        & C_{47} & C_{48} \\
					     &       &        &        & I      & C_{56} &        & -C_{48} \\ \hline
					     &       &        &        &        & D_6    &        &  \\
					     &       &        &        &        &        & D_7    &  \\
					     &       &        &        &        &        &        &  D_8
			\end{array} 
	\right) \;\;\; 
$$

$$ \begin{array}{ll} D_1, D_2, D_3& \in \SL_{2m}(\Z) \\
					D_6, D_7, D_8& \in \SL_{3m}(\Z) \\
					C_{14}, C_{24}, C_{35} &\in \Mat_{2m\times m}(\Z) \\
					C_{16}, C_{17}, C_{27}, C_{28}, C_{36}, C_{38} &\in \Mat_{2m \times 3m}(\Z)\\
					C_{47}, C_{48}, C_{56} &\in \Mat_{m \times 3m}(\Z)
\end{array} 
$$

We justify some of the less obvious claims hidden in the above description. Any $\varphi \in \mc G$
is the restriction of some $\varphi \in \tilde{\mc G}$ to $H_{\mcO}$. Then,
$\varphi(e_{\mu_i}) = e_{\mu_i} + z_i$ where $z_i \in H_{\mcO}[W_i]$. For $2 \leq i \leq 3$,
$$\varphi(e_{\mu_i} - e_{\mu_1}) = e_{\mu_i} - e_{\mu_1} + z_i - z_1.$$ 
Since $H_{\mcO}[W_i]$ is generated
by $E_i$, we see that the column for $e_{\mu_i}-e_{\mu_1}$ will have $I$ on the diagonal. Since $-e_{\mu_1}$ appears in both
vectors of $F$, the submatrix $C_{14}$ appears twice.
A basis vector in $E_6$ has the form $e_v - e_{\mu_3}$. We know $\varphi(e_v - e_{\mu_3}) = z$
for some $z \in  H_{\mcO}[U(W_6)]$ which has basis $E_2, E_3, e_{\mu_3} - e_{\mu_2}, E_6$. However, in our chosen
basis, $e_{\mu_3} - e_{\mu_2} = e_{\mu_3} - e_{\mu_1} - (e_{\mu_2} - e_{\mu_1})$, and so the $C_{48}$ repeats
with a change in sign as indicated. One can check via similar analyses that there are no other dependencies among the coefficients
in the matrix.

\subsection{Elementary lemmas on reduced norms}
We present a few lemmas verifying that the reduced norm behaves in certain ways like determinant. While these results are almost
certainly known, we provide proofs for convenience. We will require the use of the reduced characteristic polynomial which we
define similarly to the reduced norm. Given a finite-dimensional central simple $K$-algebra $A$, there exists a field extension
$L \geq K$ and an isomorphism of $L$-algebras $\psi: A \otimes_L K \cong \Mat_n(L)$ for some $n$. The reduced characteristic
polynomial of $a \in A$, denoted $\Prd(a)$, is the characteristic polyonmial of $\psi(a)$. This is well known to be a polynomial
in $K[x]$ and to be independent of the choice of $L$ and $\psi$. It's clear from the definitions that the reduced norm is the product
of the roots of the reduced characteristic polynomial.

\begin{lemma} \label{lemma:nrdonparabolic}
	Let $B$ be a finite-dimensional simple $\Q$-algebra with center $K$, and let $m \in \N$. Let $M$ be a free summand of the free $B$-module $B^m$,
	and let $\End_B(B^m, M)$ be the subalgebra of $\End_B(B^m)$ consisting of endomorphisms which preserve $M$. Let 
	$P: \End_B(B^m, M) \to \End_B(B^m/M)$ and $r: \End_B(B^m, M) \to \End_B(M)$ be the projection and restriction maps. Then,
	for any $f \in \End_B(B^m, M)$, we have $\nrd(f) = \nrd(P(f)) \cdot \nrd(r(f))$.
\end{lemma}
\begin{proof}
	Let $f \in \End_B(B^m, M)$. Let $M'$ be a complement of $M$, and let $g \in \End_B(B^m, M)$ be an endomorphism preserving $M$
	and $M'$ satisfying $P(f) = P(g)$ and $r(f) = r(g)$. Let $h = \id_{B^m} - f g^{-1}$. Then $h(M) = 0$, and $h(B^m) \subseteq M'$.
	I.e. $h^2 = 0$. This implies that $(1-x)^2$ divides $\Prd(fg^{-1})$, and so $\nrd(fg^{-1}) = 1$. Consequently $\nrd(f) = \nrd(g)$.
	
	We have $g$ in the subalgebra $\End_B(M)\times \End_B(M') \subseteq \End_B(B^m)$. Since $P$ induces an
	isomorphism $\End_B(M') \cong \End(B^m/M)$,
	it suffices to prove $\nrd(g) = \nrd(g|_M)\cdot \nrd(g|_{M'})$ as reduced norm is invariant under isomorphisms of simple algebras.
	Choose isomorphisms $\End_B(M) \otimes_K L \cong \Mat_{n}(L)$ and $\End_B(M) \otimes_K L \cong \Mat_{n'}(L)$. These induce
	an isomorphism $(\End_B(M)\times \End_B(M'))\otimes_K L \cong \Mat_{n}(L) \times \Mat_{n'}(L)$. Choose an
	isomorphism $\End_B(B^m) \cong \Mat_{n+n'}(L)$. The inclusion $\End_B(M)\times \End_B(M') \hookrightarrow \End_B(B^m)$ then 
	leads to an embedding $\Mat_{n}(L) \times \Mat_{n'}(L) \to \Mat_{n+n'}(L)$. The lemma would be proved if this were
	the canonical embedding or conjugate to the canonical embedding. This is established by the next lemma.
\end{proof}

\begin{lemma}
	Let $L$ be a field extension of $\Q$. Then there is one embedding $\Mat_{n}(L) \times \Mat_{n'}(L) \to \Mat_{n+n'}(L)$
	of $L$-algebras up to conjugation in $\Mat_{n+n'}(L)$.
\end{lemma}
\begin{proof}
	Let $f : \Mat_{n}(L) \times \Mat_{n'}(L) \to \Mat_{n+n'}(L)$ be any such embedding. Then, $f$ restricts to an embedding of $L$-algebras
	$g: \Mat_{n}(L) \to \Mat_{n+n'}(L)$. Since these are both simple $L$-algebras, the Skolem--Noether Theorem implies that 
	$g$ is conjugate by an element of $a \in \Mat_{n+n'}(L)$ to the canonical embedding which sends $\Mat_{n}(L)$ to the upper left
	block in $\Mat_{n+n'}(L)$. Let $f'(c) = a f(c) a^{-1}$. Then, the fact that the images $f'( \Mat_{n}(L))$ and $f'( \Mat_{n'}(L))$
	commute and considerations of dimension imply that $f'$ embeds $\Mat_{n}(L) \times \Mat_{n'}(L)$ as block diagonal matrices
	of the appropriate size. Applying Skolem-Noether again to each block, we can apply a further conjugation to get the canonical
	embedding $\Mat_{n}(L) \times \Mat_{n'}(L) \hookrightarrow \Mat_{n+n'}(L)$.
\end{proof}

\bibliographystyle{abbrv}
\bibliography{AutRAAG}

\end{document}